\documentclass[a4paper, 12pt]{amsart}
\usepackage{amsmath}
\usepackage{amssymb, latexsym, amsthm}
\usepackage[all]{xy}

\newtheorem{Thm}{Theorem}[section]
\newtheorem{Prop}[Thm]{Proposition}
\newtheorem{Cor}[Thm]{Corollary}
\newtheorem{Lem}[Thm]{Lemma}

\theoremstyle{definition}
\newtheorem*{definition}{Definition}
\newtheorem*{remark}{Remark}

\numberwithin{equation}{section}

\begin{document}

\newcommand{\Coim}{\mathrm{Coim}}
\newcommand{\Z}[0]{\mathbb{Z}}
\newcommand{\Q}[0]{\mathbb{Q}}
\newcommand{\F}[0]{\mathbb{F}}
\newcommand{\N}[0]{\mathbb{N}}
\renewcommand{\O}[0]{\mathcal{O}}
\newcommand{\p}[0]{\mathfrak{p}}
\newcommand{\m}[0]{\mathrm{m}}
\newcommand{\Tr}{\mathrm{Tr}}
\newcommand{\Hom}[0]{\mathrm{Hom}}
\newcommand{\Gal}[0]{\mathrm{Gal}}
\newcommand{\Res}[0]{\mathrm{Res}}
\newcommand{\id}{\mathrm{id}}
\newcommand{\cl}{\mathrm{cl}}
\newcommand{\mult}{\mathrm{mult}}
\newcommand{\adm}{\mathrm{adm}}
\newcommand{\tr}{\mathrm{tr}}
\newcommand{\pr}{\mathrm{pr}}
\newcommand{\Ker}{\mathrm{Ker}}
\newcommand{\ab}{\mathrm{ab}}
\newcommand{\sep}{\mathrm{sep}}
\newcommand{\triv}{\mathrm{triv}}
\newcommand{\alg}{\mathrm{alg}}
\newcommand{\ur}{\mathrm{ur}}
\newcommand{\Coker}{\mathrm{Coker}}
\newcommand{\Aut}{\mathrm{Aut}}
\newcommand{\Ext}{\mathrm{Ext}}
\newcommand{\Iso}{\mathrm{Iso}}
\newcommand{\M}{\mathcal{M}}
\newcommand{\GL}{\mathrm{GL}}
\newcommand{\Fil}{\mathrm{Fil}}
\newcommand{\an}{\mathrm{an}}
\renewcommand{\c}{\mathcal }
\newcommand{\W}{\mathcal W}
\newcommand{\R}{\mathcal R}
\newcommand{\crys}{\mathrm{crys}}
\newcommand{\st}{\mathrm{st}}
\newcommand{\CM}{\mathrm{CM\Gamma }}
\newcommand{\CV}{\mathcal{C}\mathcal{V}}
\newcommand{\G}{\mathrm{G}}
\newcommand{\Map}{\mathrm{Map}}
\newcommand{\Sym}{\mathrm{Sym}}
\newcommand{\Spec}{\mathrm{Spec}}
\newcommand{\Gr}{\mathrm{Gr}}
\newcommand{\I}{\mathrm{Im}}
\newcommand{\Frac}{\mathrm{Frac}}
\newcommand{\LT}{\mathrm{LT}}
\newcommand{\Alg}{\mathrm{Alg}}
\newcommand{\MG}{\mathrm{M\Gamma }}
\newcommand{\To}{\longrightarrow}
\newcommand{\md}{\mathrm{mod}}
\newcommand{\MF}{\mathrm{MF}}
\newcommand{\CMF}{\mathcal{M}\mathcal{F}}
\newcommand{\C}{\mathcal}
\newcommand{\Aug}{\mathrm{Aug}}
\renewcommand{\c}{\mathcal }
\newcommand{\uL}{\underline{\mathcal L}}
\newcommand{\Md}{\mathrm{Md}}
\newcommand{\wt}{\widetilde}

\title[Group schemes of period 2]
{Group schemes of period 2}
\author{Victor Abrashkin}
\address{Department of Mathematical Sciences, 
Durham University, Science Laboratories, 
South Rd, Durham DH1 3LE, United Kingdom \ \&\ Steklov Mathematical 
Institute, Gubkina str. 8, 119991, Moscow, Russia
}
\email{victor.abrashkin@durham.ac.uk}
\date{Oct 31, 2011}

\begin{abstract} We give an explicit construction of the antiequivalence 
of the category of finite flat commutative group 
schemes of period 2 defined over a 
valuation ring of a 2-adic field with 
algebraically closed residue field. 
This result extends the earlier author's approach 
to group schemes of period $p>2$ 
from Proceedings LMS, 101, 2010, 207-259.
\end{abstract}
\maketitle

\section*{Introduction} \label{S0}

\subsection{Basic notation}\label{S0.1} 

Everywhere in the paper 
$k$ is algebraically closed field of characteristic $2$, 
$K_{00}$ is the fraction field of the ring of Witt 
vectors $W(k)$ and $[K_0:K_{00}]=e\in\N $.   
Let  
$O_0=O_{K_0}$ be the valuation ring of $K_0$, $\pi $ --- a fixed uniformiser
in $K_0$, $K=K_0(\pi )$, where $\pi ^2=\pi _0$, and $O=O_K$. 
We set 
$S=k[[t]]$ where $t$ is a variable. Let $\sigma :S\longrightarrow S$ 
be such that $\sigma (s)=s^2$, $s\in S$. Denote by 
$\kappa _{SO}:S/t^{2e}\To O/2O$  the rings isomorphism such that 
$\kappa _{SO}|_k=\id $  and $\kappa _{SO}:t\operatorname{mod}
t^{2e}\mapsto\pi\operatorname{mod}2$. 

For a natural number $u$, $\underline{i}$ denotes always a vector 
of length $u$ with coordinates from the set $\{0,1\}$ and 
$r(\underline{i})$ denotes the sum of these coordinates.

\subsection{Categories of group schemes}\label{S0.2}

Let $R$ be a local ring of characteristic 0 with residue field $k$. 
Denote by $\Gr _R$ the category of 
finite flat commutative group schemes $G$ over 
$R$ such that $2\id _{G}=0$. 

Recall that $G=\Spec A(G)$, where $A(G)$ is a flat $R$-algebra of finite 
rank $|G|$ and the structure of group scheme on $G$ is given via the 
$R$-algebra morphisms $e_G:A(G)\To O$ (counit) and 
$\Delta _G:A(G)\To A(G)\otimes _{R}A(G)$ (coaddition) 
satisfying standard axioms. 

Denote by $\Gr ^{et}_{R}$ and $\Gr _R^{mult}$ the full subcategories 
in $\Gr _R$  
of etale and, resp., multiplicative group schemes. Then any $G\in\Gr _R$ 
has the maximal etale quotient  
$j^{et}:G\To G^{et}$ and the maximal multiplicative subobject  
$i^{mult}:G^{mult}\To G$. 

Because $k$ is algebraically closed  
any etale object in $\Gr _R$ is a product of finitely many 
copies of the constant etale group scheme of order 2, 
$(\Z /2)_{O}=\Spec\,\Map (\Z /2, R)$. Similarly, any multiplicative 
group scheme in $\Gr _R$  is a product of finitely many copies of 
the constant multiplicative group scheme of order 2, 
$\mu _{2}=\Spec R[\Z /2]$. 

Introduce the 
category $\Gr _R^*$ as follows. Its objects are the objects from 
$\Gr _R$ and for any $G_1,G_2\in\Gr _R$, 
$$\Hom _{\Gr _{R}^*}(G_1,G_2)=\Hom _{\Gr _R}(G_1,G_2)/\C R(G_1,G_2)$$
where $\C R(G_1,G_2)$ consists of the morphisms 
$$G_1\overset{j^{et}}\To G_1^{et}\overset{f}\To 
G_2^{mult}\overset{i^{mult}}\To G_2$$ 
whith arbitrary $f\in\Hom _{\Gr _R}(G_1^{et}, G_2^{mult})$. 
Note that $\Hom _{\Gr _R}((\Z /2)_R,\mu _{2,R})$ has only one 
non-trivial element given by the embedding of $R$-algebras 
$$R[\Z /2]=R\bar 0+R\bar 1\To \Map (\Z /2,R)=R\oplus R$$
such that $\bar 0\mapsto (1,1)$ and $\bar 1\mapsto (1,-1)$.

\subsection{Categories of filtered modules}\label{S0.3}

Let $\CMF _S$ be the category of the triples $(M^0,M^1,\varphi _1)$ 
such that 
$M^1\subset M^0$ are $S$-modules and 
$\varphi _1:M^1\longrightarrow M^0$ 
is a $\sigma $-linear  morphism of $S$-modules. 
The morphisms in $\CMF _S$ are compatible 
morphisms of $S$-modules commuting with $\varphi _1$.

Denote by 
$\MF ^e_S$ the full subcategory in $\CMF _S$ 
consisting of the triples $(M^0,M^1,\varphi _1)$
such that 

$\bullet $\ $M^0$ is a free $S$-module of finite rank;

$\bullet $\ $M^1\supset t^eM^0$;

$\bullet $\ $\varphi _1(M^1)S=M^0$. 

\ \

The full subcategory of etale filtered modules 
$\MF ^{e,et}_S$ in $\MF ^e_S$ 
consists of 
$(M^0,M^1,\varphi _1) $ such that  $M^1=t^eM^0$. 
One can see easily that any 
$\C M=(M^0,M^1,\varphi _1)
\in\MF ^e_S$ has a unique maximal etale subobject 
$i^{et}:\C M^{et}=(M^{0,et}, t^eM^{0,et},\varphi _1)\To\C M$.  
Suppose $\varphi _0:M^0\To M^0$ is such that 
$\varphi _0(m):=\varphi _1(t^em)$ for any $m\in M^0$. 
Then $\bar M^{0,et}:=M^{0,et}\otimes _Sk$ is the maximal 
$k$-submodule of $\bar M^0:=M^0\otimes _Sk$ such that  
$\varphi _0$ induces an invertible $\sigma $-linear automorphism 
on $\bar M^{0, et}$. 
Notice that $i^{et}$ can be included into the following 
short exact sequence 
$$0\To \C M^{et}\overset{i^{et}}\To 
\C M\To \C M^{loc}\To 0$$
where $\C M^{loc}=(M^{0,loc}, M^{1,loc},\varphi _1)\in\MF ^e_S$ and 
$\bar M^{0,loc}:=M^{0,loc}\otimes _Sk$ can be  
naturally indentified with 
the maximal $k$-submodule in $\bar M^0$ such that 
$\varphi _0$ induces its nilpotent 
endomorphism.  
Any etale filtered module is a 
direct sum of finitely many copies of 
$\C S^{et}:=(Sm, St^em,\varphi _1)\in\MF ^{e,et}_S$, 
where $\varphi _1(t^em)=m$.

The full subcategory of multiplicative filtered modules 
$\MF ^{e,mult}_S$ in $\MF ^e_S$ 
consists of $(M^0,M^1,\varphi _1)$ such that  $M^1=M^0$. 
Any $\C M\in\MF ^e_S$ has a unique maximal 
multiplicative quotient 
$j^{mult}:\C M\To\C M^{mult}$. Introduce the morphism 
$\psi _0:\bar M^0\To \bar M^0$ as follows: if $m\in M^1$ and 
$\varphi _1(m)=n\in M^0$ then $\psi _0(n\otimes 1)=
m\otimes 1$. One can verify that $\psi _0$ is 
a well-defined $\sigma ^{-1}$-linear morphism of $k$-modules 
and $\bar M^{0,mult}:=M^{0,mult}
\otimes _Sk$ can be identified with 
the maximal $k$-submodule in $\bar M^0$ 
such that $\psi _0|_{\bar M^{0,mult}}$ is invertible. 
Note that $j^{mult}$ can be included into the following short exact sequence 
in the category $\MF ^e_S$, 
$$0\To \C M^u\To \C M\overset{j^{mult}}\To \C M^{mult}\To 0$$
where $\C M^u=(M^{0,u},M^{1,u},\varphi _1)\in\MF ^e_S$ and 
$\bar M^{0,u}:=M^{0,u}\otimes _Sk$ 
is the maximal $k$-submodule such that $\psi _0|_{\bar M^{0,u}}$ is nilpotent.  
Any multiplicative filtered 
module is a direct sum of finitely many copies of 
$\C S^{mult}:=(Sn, Sn, \varphi _1)\in\MF ^{e,mult}_S$, 
where $\varphi _1(n)=n$.

Introduce the category $\MF ^{e*}_S$ as follows. 
Its objects are the objects of 
$\MF ^e_S$ and for any $\C M_1,\C M_2\in\MF ^e_S$, 
$$\Hom _{\MF ^{e*}_S}(\C M_1,\C M_2)=
\Hom _{\MF ^e_S}(\C M_1,\C M_2)/R(\C M_1,\C M_2)$$
where $R(\M _1,\M _2)$ consists of the morphisms of $\MF ^e_S$ 
of the form 
$$\M _1\overset{j^{mult}}
\To \M_1 ^{mult}\overset{f}\To \M_2 ^{et}\overset{i^{et}}\To \M _2$$
with arbitrary $f\in\Hom _{\MF ^e_S}(\C M_1^{mult},\C M_2^{et})$. 

Note that $\Hom _{\MF ^e_S}(\C S^{mult},\C S^{et})$ has only one 
non-trivial morphism  and it is given by the correspondence 
$n\mapsto t^{2e}m$. 
\medskip 

\subsection{Main result} 

In this paper we prove the following theorem.

\begin{Thm} \label{T0.1}
There is an antiequivalence of categories 
$$\C F^{O*}_{O_0}:\Gr ^*_{O_0}\To\MF ^{e*}_S$$
\end{Thm}

For $p>2$,  there is an antiequivalence of categories 
$\C F^O_{O_0}:\Gr _{O_0}\To \MF ^e_S$. This was proved 
by C.Breuil  \cite{Br} and M.Kisin \cite{Ki1,Ki2}  
in a more general context of all $p$-group schemes. 
The proofs are obtained from the study of $p$-divisible groups and 
essentially use the crystalline Dieudonne theory which is built 
on a geometrical approach due to the Raynaud theorem about 
the existence of embedding of any $p$-divisible group into an abelian scheme.  
This approach has been generalized recently by W.Kim \cite{Kim}, 
E.Lau \cite{Lau} and T.Liu \cite{Liu} to the case $p=2$. 
(Lau's result uses Zink's theory of displays and windows.)

On the other hand, an explicit 
and direct construction of the antiequivalence $\C F^O_{O_0}$ 
in the case $p>2$ 
was given by the author \cite{Ab}. The above theorem  
extends that construction to the case $p=2$. We should notice that 
this extension is very far from to be straightforward 
for the following reasons.  

First, when relating group schemes over $O_0\subset O$ and filtered 
$S$-modules we use the identification of rings $S/t^{2e}$ and $O/2$. 
But when working modulo $2$ we can't control quite efficiently all 
morphisms in the category $\Gr _O$, e.g.  both the elements 
of $\Hom _{\Gr _O}((\Z /2)_O, \mu _{2,O})$ coincide modulo $2O$. 
This explains why we are forced to use the quotient categories  
$\Gr ^*_O$, $\Gr ^*_{O_0}$ and $\MF ^{e*}_S$. On the other hand, 
the above example represents essentially the only aspect we are 
losing in our approach and Theorem \ref{T0.1} gives 
essentially complete information about the objects and morphisms 
of the category $\Gr _{O_0}$.  

Second, when working with odd prime numbers $p>2$ the ideal $pO$ 
is provided with nilpotent $DP$-structure and 
is considerably smaller than the maximal $DP$-ideal in $O$ if $e>1$. 
If $p=2$ we have no such ``safety margin'' because $2O$ is already 
the maximal $DP$-ideal in $O$. The adjustment of methods 
of \cite{Ab} to the case $p=2$ required a profound revision of all constructions 
used in there, especially the proof of the surjectivity 
of the functor $\c G_{O_0}^O$. In particular, 
so-called Main Lemma was restated in a more precise form and 
provided with an essential elaboration. As a result, 
all main features of our approach from \cite{Ab} 
were preserved in the case $p=2$.

Finally, notice that all applications developed in \cite{Ab}: 
a criterion for the Galois module to come from the Galois module of 
geometric points of $G\in\Gr _{O_0}$, the relation between 
group
 schemes from $\Gr _{O_0}$ and Faltings's strict modules in characteristic 2 
and an explicit description of the duality in $\Gr _{O_0}$ can be 
done along the lines of the approach from \cite{Ab} in the case $p=2$ as well. 
\medskip 

\subsection{Brief description of used methods} 

Let $\Aug _O$ be the category of augmented $O$-algebras. 
Introduce the equivalence relation $R$ such that for 
$f_1,f_2\in\Hom _{\Aug _O}(\c B_1,\c B_2)$, $f_1\underset{R}\sim f_2$ iff 
$f_1$ and $f_2$ coincide 
modulo some DP-ideal in $\c B_2$ (cf. Subsection \ref{S1.3} 
for the definition of this ideal). Denote by $\Aug ^*_O$ 
the category whose objects are the objects of $\Aug _O$ but the 
morphisms are the $R$-equivalence classes of morphisms in $\Aug _O$. 
In our approach the category $\Aug _O$ (resp, $\Aug ^*_O$) relates the categories 
$\MF _S^{e}$ and $\Gr _O$ (resp., $\MF _S^{e*}$ and $\Gr _O^*$). 

As first step we associate with $\c M=(M^0,M^1,\varphi _1)\in\MF ^e_S$ 
a set of augmented $O$-algebras $\Aug _O(\c M)$ each 
of whose members  is constructed 
after choosing an appropriate special basis for $M^0$ 
and a couple of other choices. One also defines 
for any $\c A\in\Aug _O$, a canonical object $\iota (\c A)\in\CMF _S$ and if 
$\c A\in\Aug _O(\c M)$ then we have a natural map 
$\iota _{\c M}:\c M\To \iota (\c A)$ in $\CMF _S$. Then 
one observes that the correspondences $\c A\mapsto\Aug _O(\c M)$ and 
$\c B\mapsto \iota (\c B)$ are left-adjoint. This results 
in the following property: if $\c N\in \MF ^e_S$ and $\c B\in\Aug _O(\c N)$ 
then we have natural identifications 
$$\Hom _{\MF ^{e*}_S}(\c M,\c N)=\Hom _{\CMF _S}
(\iota _{\c M}(\c M),\iota _{\c N}(\c N))
\subset \Hom _{\Aug ^*_O}(\c A,\c B). 
$$
This allows us to show that $\c A$, as an object of $\Aug ^*_O$, 
is functorial in $\c M$ (viewed as an object of $\MF ^{e*}$), and that 
the assignment $\c M\mapsto \c A$ is functorial. It also allows us 
to define a family 
$\c L_{\c A}$ of Hopf algebra structures on $\c A$ whose spectrums 
are group schemes over $O$ that are isomorphic in $\Gr ^*_O$. 
(Use the diagonal embedding of $\c M$ into $\c N=\c M\oplus\c M$.) 
In this way we obtain 
a faithful functor $\c G_O:\MF ^{e*}_S\To\Gr _O^*$. 

This functor is actually full. To see this, we describe $\iota _{\c M}(\c M)$ 
in terms of the Hopf algebra structure on $\c A$, and so we find 
that it is an object intrinsically attached to $\c G_O(\c M)$. 
So we obtain the following commutative diagram: 
$$
\xymatrix 
{&\Hom _{\Aug _O^*}(\c A,\c B)
\ar[rr]^{\simeq }_{\text{Prop. \ref{P1.7}}}
&&\Hom _{\CMF _S}(\iota _{\c M}(\c M),\iota (\c B))&
\\
&&&&&\\
&\Hom _{\Gr ^*_O}(\c G_O(\c N),\c G_O(\c M))\ar@{^{(}->}[uu]
\ar[rr]^{\simeq}_{\text{Prop. \ref{P1.13}}}
&&\Hom _{\CMF _S}(\iota _{\c M}(\c M),\iota _{\c N}(\c N))\ar@{^{(}->}[uu]&
\\
&&&&&\\
&{}&&\Hom _{\MF ^{e*}_S}(\c M,\c N)\ar[lluu]\ar[uu]^{\simeq}&}
$$ 
These ideas have been applied earlier 
by the author in \cite{Ab2, Ab3} to describe 
the category 
$\Gr ^*_O$, where $O$ is the valuation ring of a field extension 
of $\Q _p$ with small ramification. 
 
Next, one shows that $\c G_O$ factors through a functor $\c G_{O_0}^{O}:
\MF ^{e*}_S\To \Gr ^*_{O_0}$. This amounts to showing that every group scheme 
$\c G_O(\c M)$ descends  to one over $O_0$. To accomplish this one uses tame 
descent and induction on rank of $\c M$. This goes more or less along 
the lines of \cite{Ab}. 

All that remains to do is to show that $\c G_{O_0}^O$ is essentially surjective. 
One again uses tame descent and induction, this time on $|G_0|=p^s$. 
This is most difficult part of the paper, where we need essential elaboration of 
Main Lemma from \cite{Ab} and where a special role of prime number $p=2$ 
can be explained in the following way. When applying induction on 
$s$ we present $G=G_0\otimes _{O_0}O$ as an extension of a group scheme 
$H$ of order $p^{s-1}$ 
via a group scheme of order $p$. Then on the level of algebras 
$A(G)$ and $A(H)$ of the group schemes $G$ and $H$, 
the Kummer theory provides us with a class of generators $\theta $ such that 
$A(G)=A(H)[\theta ]$. Main Lemma allows us to make a very special choice of 
$\theta $ and then via the 
Lubin-Tate logarithm $l_{LT}(X)=X+X^2/p+X^{p^2}/p^2+\dots $, 
we can relate 
$G$ with a group scheme of the form $\c G_O(\c M)$, $\c M\in\MF ^e_S$. 
This special choice becomes much more delicate in the case $p=2$ because when $p>2$ 
it was enough to use only the first two terms of the above expansion of $l_{LT}$ but in 
the case $p=2$ we need to take into account one term more. 
\medskip

\section{Construction of the functor $\C G_O:\MF ^e_S\To \Gr _{O}$}
\label{S1}

We use all notation and assumptions from Introduction, in particular, 
the definition of the category $\MF ^e_S$ and 
the corresponding properties of its objects. 
Remind also that we fixed an identification 
$\kappa _{SO}:S/t^{2e}S\simeq O/2O$ such that 
$\kappa _{SO}|_k=\id $ and 
$\kappa _{SO}(t\operatorname{mod}t^{2e})
=\pi \operatorname{mod}2$.

\subsection{$\varphi _1$-lifts}
\label{S1.1}
Suppose 
$\C M=(M^0,M^1, \varphi _1)\in\MF ^e_S$, 
$\C N=(N^0,N^1, \varphi _1)\in\CMF _S$ and 
$\theta\in \Hom _{\CMF _S}(\C M, \C N)$.   

\ \ 

\begin{definition}  a) $\theta $ is a  
$\varphi _1${\it -lift} if $\theta (M^0)=N^0$, $\theta (M^1)=N^1$ and  
$\Ker\,\theta =\Ker\theta |_{M^1}:=T$;
b) $\varphi _1$-lift $\theta $ is {\it nilpotent} if $\varphi _1|_T$ 
is topologically nilpotent, i.e. $\bigcap _{n}\varphi _1^n(T)=0$;
c) $\varphi _1$-lift $\theta $ is {\it special} 
if there is a submodule $T'$ of $T$ such that 
$T=T'+(T\cap t^{2e}M^{0,et})$ and  
$\varphi _1$ induces a topologically nilpotent endomorphism of $T'$. 
\end{definition}

\begin{Prop}\label{P1.1}
 Suppose for $i=1,2$, $\theta _i\in\Hom _{\CMF _S}
(\c M_i,\c N_i)$ are $\varphi _1$-lifts and $h\in\Hom _{\CMF _S}(\c N_1,\c N_2)$. 
Then the set $\c L(h)$ of all $f\in\Hom _{\MF ^e_S}(\c M_1,\c M_2)$ such that 
$\theta _2\circ f=h\circ\theta _1$ is not empty. If in addition:

{\rm a)} $\theta _2$ is nilpotent then $\c L(h)$ consists only of one element;
\medskip 

{\rm b)} $\theta _2$ is special then all elements of $\c L(h)$ belong to 
$R(\c M_1,\c M_2)$, i.e. coincide in $\MF ^{e*}_S$. 
\end{Prop}

\begin{proof} We start with the following Lemma. 
 
\begin{Lem} \label{L1.1d} Suppose $L$ is a finitely generated  
$S$-module and $\varphi $ is a 
$\sigma $-linear operator on $L$. Then the operator $\id _L-\varphi $ is epimorphic. 
If, in addition, $\varphi $ is topologically nilpotent then $\id _L-\c A$ is bijective. 
\end{Lem} 

\begin{proof} [Proof of Lemma] Part b) is obvious. 
In order to prove a) notice first that we can 
replace $L$ by $L/tL$ and, therefore, assume that 
$L$ is a finite dimensional vector space over $k$.  
Then there is a decomposition $L=L_1\oplus L_2$, where $\varphi $ 
is invertible on $L_1$ and nilpotent on $L_2$. 
It remains to note that $L_1=L_{0}\otimes _{\F _p}k$, where 
$L_0$ is a finite dimensional $\F _p$-vector space such that $\varphi |_{L_0}=\id $. 
The existence of $L_0$ is a standard fact of $\sigma $-linear algebra: 
if $s=\dim _kL_1$ and $A\in M_s(k)$ is a matrix of $\varphi |_{L_1}$ 
in some $k$-basis of $L_1$ then  
$L_0=\{(x_1,\dots ,x_s)\in k^s\ |\ 
(x_1^p,\dots ,x_s^p)A=(x_1,\dots ,x_s)\}$; the $\F _p$-linear space 
$L_0$  has dimension $s$ because 
the corresponding equations determine an etale algebra of rank $p^s$ 
over algebraically closed field $k$. The Lemma is proved. 
 \end{proof}

Now suppose for $i=1,2$, $\c M_i=(M_i^0,M_i^1,\varphi _1)$ and 
$\c N_i=(N_i^0,N_i^1,\varphi _1)$. Let a vector $\bar m_1\in (M_1^0)^s$ and a 
matrix $C\in M_s(S)$ be such that the coordinates of $\bar m_1$ and $\bar m_1C$
form an $S$-basis of $M_1^0$ and, resp., $M_1^1$, and 
$\varphi _1(\bar m_1C)=\bar m_1$. 

If $\bar n_2=(h\circ \theta _1)\bar m_1$ then $\varphi _1(\bar n_2C)=\bar n_2$. 
Choose a vector $\bar m_2\in (M_1^0)^s$ such that $\theta _2(\bar m_2)=\bar n_2$. 
Then $\bar m_2-\varphi _1(\bar m_2C)=\bar t_2\in T_2^s$, where $T_2=\Ker\,\theta _2$. 
The elements of $\c L(h)$ correspond to vectors $\bar t\in T_2^s$ such that 
$\bar m_2+\bar t=\varphi _1((\bar m_2+\bar t)C)$ or equivalently 
$\bar t-\varphi _1(\bar tC)=-\bar t_2$. 

By above Lemma such $\bar t$ 
always exists and, therefore, $\c L(h)\ne\emptyset $. 
If $\theta _2$ is nilpotent then such $\bar t$ is unique and part a) is proved.

Now suppose $\theta _2$ is special. Then $T_2=T_2'+(T_2\cap t^{2e}M_2^{0,et})$ and 
$\varphi _1|_{T_2'}$ is topologically nilpotent. Clearly, we can assume that 
$T_2\cap t^{2e+1}M_2^{0,et}\subset T_2'$. 

Note that $\c N_2=(M_2^0/T_2, M_2^1/T_2,\varphi _1)$ and $\theta _2$ 
appears as the composition of two natural projections in $\CMF _S$ 
$$\c M_2\overset{\alpha }\To \c N_2':=(M_2^0/T_2', M_2^1/T_2',\varphi _1)
\overset{\beta }\To \c N_2.$$
Here $\alpha $ is a nilpotent $\varphi _1$-lift and, therefore,  $f\in\c L(h)$ is 
uniquely determined by $\alpha\circ f\in\Hom _{\CMF_S}(\c M_1,\c N_2')$. 

Suppose for $i=1,2$, $f_i\in\c L(h)$ and $f_i'=\alpha\circ f_i$. 
Then $f_1'-f_2'\in\Ker\,\beta _*\subset \Hom _{\CMF _S}(\c M_1,\wt{\c N}_2)$, where 
$\wt{\c N}_2=(\wt{N}_2,\wt{N}_2,\varphi _1)$, $\wt{N}_2=t^{2e}M_2^{0,et}/t^{2e+1}$ and 
$\varphi _1$ is $\sigma $-linear automorphism of $\wt{N}_2$. 

Let $\wt{\c M}_2:=(t^{2e}M_2^{0,et}, t^{2e}M_2^{0,et},\varphi _1)$. 
Clearly it is a subobject of $\c M_2^{et}$ in $\MF ^e_S$. 
Then the natural projection $\wt{\c M}_2\To\wt{\c N}_2$ is 
a nilpotent $\varphi _1$-lift and, therefore, $f_1-f_2$ factors through the embedding 
$\wt{\c M}_2\subset \c M_2^{et}$. Finally, $\wt{\c M}_2$ is a multiplicative object in 
$\MF ^e_S$ and this implies that $f_1-f_2$ factors through the natural projection 
$\c M_1\To\c M_1^{mult}$. 

The proposition is proved. 
\end{proof}

\begin{Cor} \label{C1.2} 
a) If $\theta $  is nilpotent 
then $\C M$ is defined uniquely by $\C N$ up to 
a unique isomorphism in the category $\MF ^e_S$;

b) If $\theta $ is special then 
$\C M$ is defined uniquely by $\C N$ up to 
a unique isomorphism in the category $\MF ^{e*}_S$.
\end{Cor}

\subsection{Extension of scalars} \label{S1.2}

Let $S'=S[t']$, where $t'^2=t$. If $\c M=(M^0, M^1,\varphi _1)\in\MF ^e_S$ 
then $\C M\otimes _SS':=(M^0\otimes _SS', M^1\otimes _SS', 
\varphi _1\otimes\sigma )\in\MF ^{2e}_{S'}$. (Here $\sigma (s)=s^2$ for any $s\in S'$.) 

If $\C M'=(M'^0, M'^1,\varphi _1')\in\MF ^{2e}_{S'}$ then 
$\varphi _1'(M'^1)=\{\varphi _1'(m)\ |\ m\in M'^0\}$ has a natural structure 
of $S$-module and coincides with $M^0$ if $\C M'=\C M\otimes _SS'$. 
This fact implies easily the following criterion of the existence 
of a descent of $\c M'$ to $S$. 

\begin{Prop} \label{P1.3}
 If $\c M'=(M'^0, M'^1,\varphi _1')\in\MF ^{2e}_{S'}$ then 
the following two properties are equivalent: 
\medskip 

{\rm a)}\ $\C M=(M^0, M^1,\varphi _1)\in\MF ^e_S$ is such that 
$\c M'=\c M\otimes _SS'$;
\medskip 

{\rm b)}\ if $M^0=\varphi _1'(M'^1)$, $M^1=M'^1\cap M^0$ and 
$\varphi _1=\varphi _1'|_{M^1}$ then 
\linebreak 
$\c M=(M^0,M^1,\varphi _1)\in\MF ^e_S$.  
\end{Prop}

\medskip

\subsection{The category of augmented $O$-algebras $\Aug _O$} 
\label{S1.3}

Let $\Aug _O$ 
be the category of augmented $O$-algebras 
$\C B=(B, I_B)$ such that  
$B$ is a flat $O$-algebra of finite rank and 
$I_B$ is an ideal of $B$ such that $O\simeq B/I_B$ 
via the natural map $o\mapsto o\cdot 1_B$, $o\in O$. 

We shall denote by $\C B^{et}$ the subobject 
$(B^{et}, I_{B^{et}})\in\Aug _O$ such that $B^{et}$ is the maximal 
etale subalgebra of $B$ and $I_{B^{et}}=B^{et}\cap I_B$. 

With above notation let $I_B^{loc}$ be the ideal of 
topologically nilpotent elements of $I_B$. Clearly, $I_B^{loc}\cap I_{B^{et}}=
\pi I_{B^{et}}$. Denote by $I_B(2)$ the ideal of 
all $b\in I_B$ such that $b^2\in 2I_B$ 
and by $I_B(2)^{loc}$ ---  
the ideal of all $b\in I_B$ such that 
$b^2\in  2I_B^{loc}$. Clearly, $I_B(2)=I_B(2)^{loc}+\pi ^eI_{B^{et}}$. 

Let $J_B=I_B(2)^2+\pi ^eI_B(2)$. Then $J_B=\widetilde{J}_B+2I_{B^{et}}$, where 
$\widetilde{J}_B=I_B(2)^{loc}I_B(2)+\pi ^eI_B(2)^{loc}$. 
Notice that $J_B$ is provided with the standard 
$DP$-structure by the map 
$b\mapsto -b^2/2$, $b\in J_B$, and $\widetilde{J}_B$ is the maximal 
ideal in $J_B$ where this $DP$-structure  
is topologically nilpotent. 

We shall denote by $\Aug ^*_O$ the following category. 
Its objects are the objects of the category $\Aug _O$ and 
for any $\C B_1,\C B_2\in\Aug _O$, 
$$\Hom _{\Aug ^*_O}(\C B_1,\C B_2)=\Hom _{\Aug _O}(\C B_1,\C B_2)/R,$$
where $R$ is the following equivalence relation:
\medskip  

{\it if $f_1,f_2\in \Hom _{\Aug _O}(\C B_1,\C B_2)$ then 
$f_1\underset{R}\sim f_2$ iff $f_1\equiv f_2\operatorname{mod}J_{B_2}$.}
\medskip 

\subsection{Families of augmented $O$-algebras 
$\Aug _O(\C M)$, $\C M\in\MF ^e_S$} \label{S1.4} 

Suppose $\C M=(M^0,M^1,\varphi _1)\in\MF ^e_S$ and  
the vector $\bar m^1=(m^1_1,\dots ,m^1_u)$ is such that 
its coordinates form an 
$S$-basis of $M^1$. One can verify that $\bar m^1$ 
can be chosen in such way that the following 
two conditions $\mathop{C1}$ and $\mathop{C2}$ are satisfied: 

\ \ 

$\mathop{C1}$: {\it the non-zero images of all   
$m^1_i$, $1\le i\le u$, in $M^0/tM^0$ are linearly 
independent over $k$};

\ \ 

$\mathop{C2}$: {\it $\bar m^{1}=(\bar m^{1,loc},t^e\bar m^{et})$ 
where 
\newline 
a) the coordinates of $\bar m^{et}$ form an $S$-basis 
of $M^{0,et}$ and  $\varphi _1(t^e\bar m^{et})=\bar m^{et}$;
\newline 
b) if $\bar m^{loc}=\varphi _1(\bar m^{1,loc})$ then the coordinates of 
$\bar m^{loc}\operatorname{mod}t$ form a basis of $\bar M^{0,loc}=M^{0,loc}\otimes _Sk$ 
over $k$.} 

\ \ 
 
Let $\bar m^0=(\bar m^{loc}, \bar m^{et})$. Then the coordinates of $\bar m^0$ form 
an $S$-basis of $M^0$. Denote by $U$ the $(n\times n)$-matrix with coefficients in $S$ 
such that $\bar m^1=\bar m^0U$. By condition $\mathop{C1}$,   
for appropriate $S$-matrices 
$U_1$ and $U_2$, we have 
$\bar m^{1,loc}=\bar m^{loc}U_1+\bar m^{et}(tU_2)$. 

The above chosen data: the vectors $\bar m^0, \bar m^1$ and 
the matrix $U\in M_u(S)$ 
--- completely 
describe the structure of $\C M\in\MF ^e_S$. 
Choose 
$C\in M_u(O)$ such that  $C\operatorname{mod}2=
U\operatorname{mod}t^{2e}$ with respect to 
the identification $\kappa _{SO}$. Define the  
$O$-algebra  
$A=O[\bar X]/\C I_A$, where 
$\bar X=(X_1,\dots ,X_u)$, 
$\C I_A=\C I_{A,K}\cap O[\bar X]$ and   
$\C I_{A,K}$ is the ideal in $K[\bar X]$ generated 
by the coordinates of the vector 
$(-1/2)(\bar XC)^{(2)}-\bar X$. 
(For any matrix $\alpha =(\alpha _{ij})$ we set 
$\alpha ^{(2)}:=(\alpha _{ij}^2)$.)

\begin{Prop} \label{P1.4} 
With above notation $A$ is a flat algebra of rank 
$2^n$ over $O$. 
\end{Prop} 

\begin{proof} 
Indeed, it can be deduced from condition $\mathop{C1}$ 
(similarly to Lemma 2.2.2 from 
\cite{Ab}) that 
$C^{(2)}$ divides the scalar matrix 
$2I_u$ in $M_u(O)$. (Note that $C$ divides $\pi ^eI_u$.)  
This implies that the ideal 
$\C I_A$ is generated by the coordinates of the vector  
$\bar X^{(2)}-2(\bar X+\bar V){C^{(2)}}^{-1}$,  
where $\bar V$ consists of $O$-linear combinations 
of $X_iX_j$, $1\leqslant i<j\leqslant u$. Therefore, 
there is an isomorphism of $O$-modules 
\begin{equation} \label{AE1}
A\simeq \oplus _{0\leqslant i_1,\dots, i_u\le 1}OX_1^{i_1}
\dots X_{u}^{i_u}
\end{equation} 
and $A$ is flat over $O$. 
\end{proof}

For the above introduced algebra  $A$, denote by $I_A$ the 
ideal in $A$ generated by 
the images of $X_1,\dots ,X_u$. Then $(A,I_A)\in\Aug _O$. 

\begin{definition}  Denote 
by $\Aug _O(\C M)$ the family of all augmented algebras $(A,I_A)\in\Aug _O$ 
obtained via the above procedure 
for all choices of $\bar m^1$ (which satisfy the conditions $\mathop{C1}$ 
and $\mathop{C2}$) and the corresponding matrix $C\in M_u(O)$.
\end{definition} 

\subsection{The $\varphi _1$-lift $\iota _{\C M}$}  
\label{S1.5}

Define the functor  $\iota :\Aug _O\To \CMF  _S$ via 
$$(B, I_B)\mapsto (I_B/J_B, I_B(2)/J_B,\varphi _1)$$
where $(B,I_B)\in\Aug _O$, $J_B$ was introduced in 
Subsection \ref{S1.3} and $\varphi _1$ is 
induced by the correspondences $b\mapsto -b^2/2$, $b\in I_B$.

For any $\C M\in\MF ^e_S$ and $(A,I_A)\in\Aug _O(\C M)$, 
there is a canonical morphism 
$\iota _{\C M}:\C M\To \iota ({A(\C M)})$  
in $\CMF_S$  such that 
$\bar m^0\mapsto \bar X\operatorname{mod} J_A$ and 
$\bar m^1\mapsto \bar X C\operatorname{mod}J _A$. Clearly, the image 
$\iota _{\c M}(\c M)$ is a subobject of 
$\iota (\c A(\c M))$ in the category $\CMF _S$. 

\begin{Prop} \label{P1.5} The map $\iota _{\C M}:\c M\To \iota _{\c M}(\c M)$ 
is a special $\varphi _1$-lift.
\end{Prop}  

\begin{proof} Consider $\widetilde{\C M}=\C M\otimes _SS/t^{2e}\in\CMF _S$. 
Clearly, the natural projection $\C M\To \widetilde{\C M}$ is 
a special $\varphi _1$-lift and $\iota _{\C M}$ is 
the composition of this projection 
and a unique  $\tilde\iota _{\C M}\in\Hom _{\CMF _S}
(\widetilde{\C M}, \iota _{\C M}(\C M))$. 

Let $\widetilde{\C M}=(\widetilde{M}^0,\widetilde{M}^1, \varphi _1)$. Then 
$$\widetilde{M}^0=\left\{ \sum o_i\widetilde{X}_i\ |\ o_1,\dots ,o_n\in O, 
\widetilde{X}_i=X_i\operatorname{mod}2I_A\right\}$$
$$\widetilde{M}^1=\left\{ \sum o_i\widetilde{Y}_i\ |\ o_1,\dots ,o_n\in O, 
(\widetilde{Y}_1,\dots ,\widetilde{Y}_n)=
(\widetilde{X}_1,\dots ,\widetilde{X}_n)C\right\}$$
where the $S$-module structure is induced by the given $O$-module structure 
via the identification $\kappa _{SO}$ and 
$\varphi _1:\widetilde{M}^1\To \widetilde{M}^0$ is given via the correspondence 
$\sum o_i\widetilde{Y}_i\mapsto\sum o_i^2\widetilde{X}_i$. 

Suppose $\tilde\iota _{\C M}=(\tilde\iota ^0_{\C M}, \tilde\iota ^1_{\C M})$. 
Then 
$$\widetilde{T}^0:=\Ker \,\tilde\iota ^0_{\C M}=
\left\{\sum o_i\widetilde{X}_i\ |\ \sum o_iX_i\in J_A\right\}$$
$$\widetilde{T}^1:=\Ker \,\tilde\iota ^1_{\C M}=
\left\{\sum o_i\widetilde{Y}_i\ |\ \sum o_iY_i\in J_A\right\}$$

The proposition will be proved if we show $\widetilde{T}^0=\widetilde{T}^1$, 
$\varphi _1(\widetilde{T}^0)\subset\widetilde{T}^0$ and 
$\varphi _1|_{\widetilde{T}^0}$ 
is nilpotent. 

Suppose $\tilde v=\sum o_i\widetilde{X}_i\in\widetilde{T}^0$. 
Then $\sum o_iX_i\in J_A\subset I_A(2)$ and $\sum _io_i^2X_i^2\in 2I_A$. 
Let $(G_1',\dots ,G_u')=2(\bar X+\bar V)(C^{(2)})^{-1}$, cf. the proof of 
Proposition \ref{P1.4}. Then $\sum _io_i^2G_i'\in 2I_A$ and due to the isomorphism 
(\ref{AE1}) we can follow the linear terms to obtain that 
\begin{equation} \label{AE2} 
 2(o_1^2,\dots ,o_u^2)(C^{(2)})^{-1}:=2(\alpha _1,\dots ,\alpha _n)\in 2O^u
\end{equation}

Clearly, there are $\alpha _1',\dots ,\alpha _u'\in O$ such that 
all $\alpha _i'^2\equiv\alpha _i\operatorname{mod}2O$ and (\ref{AE2}) implies that 
$$(o_1,\dots ,o_u)\equiv (\alpha _1',\dots ,\alpha _u')C\operatorname{mod}\pi ^e.$$
Therefore, $\sum _io_i\widetilde{X}_i$ is congruent modulo 
$\pi ^e\widetilde{M}^0$ to an $O$-linear combination of the coordinates of the vector 
$(\widetilde{Y}_1,\dots ,\widetilde{Y}_u)=(\widetilde{X}_1,\dots ,\widetilde{X}_u)C$. 
In other words, $\tilde v\in \widetilde{M}^1$, i.e. $\widetilde{T}^0=
\widetilde{T}^1$.

If $\tilde v=\sum o_i'\widetilde{Y}_i$, then 
$$\varphi _1(\widetilde{v})=\sum o_i^{\prime 2}\widetilde{X}_i=
-\widetilde{v}^2/2+\sum _{i,j}o'_io'_j\widetilde{Y}_i\widetilde{Y}_j\in 
J_A\operatorname{mod}2I_A$$
implies $\varphi _1(\widetilde{T}^0)\subset \widetilde{T}^0$. 
(Use that $J_A$ is a $DP$-ideal and $I_A(2)^2\subset J_A$.) 

Finally, let $\tilde v_0=\tilde v$ and for $n\geqslant 0$, 
$\tilde v_{n+1}=\varphi _1(\tilde v_n)$. 
We must prove that for $n\gg 0$, $\tilde v_n=0$. 

Let $A'=O[Y_1,\dots ,Y_n]$. Then $A'$ is an $O$-subalgebra 
in $A$ given by the equations 
$$(Y_1^2,\dots ,Y_n^2)+(Y_1,\dots ,Y_n)(2C^{-1})=0.$$
Therefore, any element $a\in A'$ can be uniquely presented in the form 
$$a=\sum _{1\leqslant i_1<\dots <i_s\leqslant u}
o_{i_1\dots i_s}(a)Y_{i_1}\dots Y_{i_s},$$
where all $o_{i_1\dots i_s}(a)\in O$. Set 
$\C L(a)=o_1(a)Y_1+\dots +o_u(a)Y_u$. 
Notice that if $I_{A'}$ is the augmentation ideal of $A'$ 
generated by $Y_1,\dots ,Y_u$ and $a\in I_{A'}^2$ then all 
$o_i(a)\equiv 0\operatorname{mod}\pi ^e$ 
(use that $2C^{-1}\equiv 0\operatorname{mod} \pi ^e$). 
This means that if $a_1,a_2\in I_{A'}$ and $a_1\equiv a_2\operatorname{mod}I_{A'}^2$ 
then $\c L(a_1)\equiv\c L(a_2)\operatorname{mod}\pi ^eI_{A'}$. 

With above notation let $v_0=\sum o_iX_i$ and for all 
$n\geqslant 0$, $v_{n+1}=-v_n^2/2$. Clearly, all $v_n\in J_A$ 
and there is an $N_0\geqslant 0$ such that $v_{N_0}\in 2I_A$. 

For $n\geqslant 0$, set $v_n^*=\C L(v_n)$ and denote by 
$\rho :A'\To A/2$ the composition of the natural inclusion 
$A'$ into $A$ and the reduction map $A\To A/2$. 

\begin{Lem} 
\label{L1.6} 
$\rho (v_n^*)\equiv 
\tilde v_n\operatorname{mod}(\pi ^e\widetilde{M}^1)$.
\end{Lem}

\begin{proof}[Proof of Lemma] Use induction on $n\geqslant 0$. 

Clearly, $v_0=v_0^*$ and $\rho (v_0)=\tilde v_0$. 

Suppose $\rho (v_n^*)\equiv \tilde v_n
\operatorname{mod}(\pi ^e\widetilde{M}^1)$.

If $v_n^*=\sum o_i^{(n)}Y_i$, where all $o_i^{(n)}\in O$, then 
$v_n=\sum o_i^{(n)}Y_i+\alpha $ with $\alpha\in I_{A'}^2$. Therefore, 
$$v_{n+1}\equiv -v_n^2/2\equiv \sum o_i^{(n)2}X_i\operatorname{mod}I_{A'}^2$$  
$$v^*_{n+1}\equiv \sum o_i^{(n)2}X_i\operatorname{mod}\pi ^eI_{A'}$$
and $\rho (v_{n+1}^*)\equiv\varphi _1(\rho (v_n^*))
\equiv\varphi _1(\tilde v_n)\operatorname{mod}(\pi ^e\widetilde{M}^1)$. 
The lemma is proved.
\end{proof} 

Finally, the above lemma implies that 
$\tilde v_{N_0}\in\pi ^e\widetilde{M}^1$ and, 
therefore, $\tilde v_{N_0+1}=\varphi _1(\tilde v_{N_0})
\in\pi ^{2e}\widetilde{M}^0=0$. 
The proposition is proved.

\end{proof}

\subsection{The maps $\Theta ^* $ and $\Psi ^*$} \label{S1.6} 

Suppose $\C B=(B,I_B)\in\Aug _O$ and $\c B^{et}=(B^{et}, I_{B^{et}})$ is 
the maximal etale subalgebra in $\C B$. Then $J_{B^{et}}=2I_{B^{et}}$ and  
$\iota (\C B^{et})=
(I_{B^{et}}/2I_{B^{et}}, \pi ^eI_{B^{et}}/2I_{B^{et}},
\varphi _1)\in\CMF _S$ admits a (unique) special 
$\varphi _1$-lift $\C E (\C B^{et})\in\MF ^{e,et}_S$. 

Introduce $\m (\C B^{et})=
(2I_{B^{et}}/2\pi I_{B^{et}}, 2I_{B^{et}}/2\pi I_{B^{et}},
\varphi _1)\in\CMF _S$, where $\varphi _1$ is 
induced (as usually) by the map $a\mapsto -a^2/2$, $a\in 2I_{B^{et}}$. 
Clearly, $\m (\C B^{et})$ admits 
a (unique) nilpotent $\varphi _1$-lift $M(\C B^{et})\in\MF ^{e,mult}_S$ 
and the identity morphism on $I_{B^{et}}$ induces the natural morphism 
$$\omega (\C B^{et}):M(\c B^{et})\To \C E(\c B^{et})$$ 
in the category $\MF ^e_S$. 

Suppose $\C M=(M^0,M^1,\varphi _1)
\in\MF ^e_S$ and $\C A=(A,I_A)\in\Aug _O(\C M)$. 
Introduce  the map 
$$\Theta :\Hom _{\Aug _O}(\C A,\C B)\To \Hom _{\CMF _S}(\C M,\iota (\C B))$$
by attaching to $F\in \Hom _{\Aug _O}(\C A,\C B)$ the morphism of 
filtered modules $\Theta (F)=\iota (F)\circ \iota _{\C M}$.

\begin{Prop} 
\label{P1.7} 
With the above notation:
\newline  
{\rm a)} $\Theta $ is surjective;
\newline     
{\rm b)} if either $\C B^{et}=\C S^{et}$ or $\C M^{mult}=0$ then 
$\theta $ is bijective;
\newline 
{\rm c)} there is a natural strict action of the group 
$\Hom _{\MF ^e_S}(\C M^{mult}, \C E(\C B^{et}))$ on 
$\Hom _{\Aug _O}(\C A,\C B)$ and   
the corresponding equivalence relation $R$ coincides 
with the equivalence relation 
from the definition of $\Aug ^*_O$ in Subsection {\rm \ref{S1.3}};
\newline 
{\rm d)}  $\Theta $ induces the  bijection  
$$\Theta ^{*}:\Hom _{\Aug ^*_O}(\C A,\C B)
\To\Hom _{\CMF _S}(\iota _{\c M}(\c M), \iota (\C B)).$$
\end{Prop}

\begin{proof} Suppose $\C A=(A,I_A)\in\Aug _O(\C M)$ 
is given via a special choice of 
vectors $\bar m^0$ and $\bar m^1$ 
with the coordinates in $M^0$ and, resp., $M^1$, and the matrix 
$C\in M_u(O)$ from Subsection \ref{S1.4}. 

\begin{Lem} 
\label{L1.8} 
Suppose $\bar b^0\in I_B^u$ 
is such that 
$(-1/2)(\bar b^0C)^{(2)}\equiv \bar b^0\operatorname{mod}J_B$.  
Let $\C L(\bar b^0)$ be the set of all $\bar b\in I_B^u$ such that 
$\bar b\equiv\bar b^0\operatorname{mod}\widetilde{J}_B$ and it holds 
$(-1/2)(\bar bC)^{(2)}\equiv\bar b\operatorname{mod}\widetilde{J}_B$. 
Then 

{\rm a)} $\C L(\bar b^0)\ne\emptyset $;

{\rm b)} if $\bar b_1,\bar b_2\in\C L(\bar b^0)$ then 
$\bar x=\bar b_1-\bar b_2\in J_B^u$ and  
$
(-1/2)(\bar xC)^{(2)}\equiv\bar x\operatorname{mod}\widetilde{J}_B
$.
\end{Lem}

\begin{proof}[Proof of Lemma] 
The vector $\bar b=\bar b^0+\bar x\in\C L(\bar b^0)$ iff 
$\bar x\in J_B^u$ 
and 
$$(-1/2)(\bar xC)^{(2)}-\bar x\equiv\bar b^0+(1/2)(\bar b^0C)^{(2)}
\operatorname{mod}\widetilde{J}_B.$$
Notice that $\C V=(J_B/\widetilde{J}_B)^n$ has a natural structure 
of a finite dimensional vector space over $k$ and the correspondence 
$\bar x\mapsto (-1/2)(\bar xC)^{(2)}$ induces a $\sigma $-linear 
morphism $\tilde\varphi :\C V\To\C V$. By Lemma \ref{L1.1d}, 
$\tilde\varphi -\id :\C V\To\C V$ 
is surjective. This proves part a). 
Part b) follows easily from the congruence 
$(\bar b_1C)^{(2)}\equiv (\bar b_2C)^{(2)}+(\bar xC)^{(2)}
\operatorname{mod}2J_BI_B(2)$. 
\end{proof}

Notice that 
$\Hom _{\CMF _S}(\iota _{\c M}(\C M), \iota (\C B))=$
$$\{\bar b^0\operatorname{mod}J_B
\ |\ \bar b^0\in I_B^u, (-1/2)(\bar b^0C)^{(2)}
\equiv \bar b^0\operatorname{mod}J_B\},$$
$$\Hom _{\Aug _O}(\C A,\C B)=\{\bar b\in 
I_B^u\ |\ (-1/2)(\bar bC)^{(2)}=\bar b\}$$
$$=\{\bar b\operatorname{mod}\widetilde{J}_B\ |\ \bar b\in I_B^u, 
(-1/2)(\bar bC)^{(2)}\equiv \bar b\operatorname{mod}\widetilde{J}_B\}$$
and the map $\Theta $ is given via $\bar b\operatorname{mod}\widetilde{J}_B
\mapsto \bar b^0\operatorname{mod}J_B$. 

Therefore, part a) of Proposition \ref{P1.7} follows from part a) 
of above Lemma. If $\bar x$ is the vector from part b) 
of above Lemma then the correspondence $\bar m\mapsto \bar x$ identifies 
$$\Hom _{\MF ^e_S}(\C M, M(\C B^{et}))=
\Hom _{\CMF _S}(\C M^{mult}, M(\C B^{et}))$$ 
with $\Hom _{\CMF _S}(\C M^{mult}, \m (\C B^{et}))$. This implies  
part b) of Proposition \ref{P1.7}. With the above notation the correspondence 
$\bar b\mapsto \bar b+\bar x$ determines the action 
of $\Hom _{\MF ^e_S}(\C M^{mult},\C E(\C B^{et}))$ 
on $\Hom _{\Aug _O}(\C A, \C B)$. One can easily verify 
that this action is strict and $\Theta $ induces  
bijection of the corresponding quotient 
$\Hom _{\Aug _O}(\C A,\C B)/R$ and 
$\Hom _{\CMF _S}(\iota _{\c M}(\C M), \iota (\C B))$. 
\end{proof}

\begin{remark} 
 a) By condition $C2$, $\bar m^0=(\bar m^{loc}, \bar m^{et})$ 
and therefore we have the appropriate presentation 
$\bar x=(\bar x^{loc}, \bar x^{et})$, where $\bar x$ is the vector from 
part b) of Lemma \ref{L1.8}. One can easily see that $\bar x^{et}=\bar 0$. 
In particular, the shifts by all above vectors $\bar x$ determine 
a strict action of 
$\Hom _{\CMF _S}(\C M^{mult}, \iota (\C B))$ on 
$I_B^u\operatorname{mod}\widetilde{J}_B$. 
\newline 
b) One can easily see that 
 if $\C M\in\MF ^{e,mult}_S$ and $\C B=\c B^{et}$ then 
$$\Hom _{\Aug _O}(\C A,\C B)=\Hom _{\CMF _S}(\C M,E(\C B^{et}));$$
\end{remark}

\begin{Cor} \label{C1.9} 
If $\C B\in\Aug _O(\C N)$ with $\C N\in\MF ^e_S$ 
then the above identification $\Theta ^*$ induces a functorial 
in both arguments embedding 
$$\Psi ^*:\Hom _{\MF ^{*e}_S}(\C M,\C N)\To 
\Hom _{\Aug ^*_O}(\C A,\C B)$$
and the correspondences $\C M\mapsto \c A$ (and $\C N\mapsto \C B$) 
induce a faithful functor 
$\Psi ^*:\MF ^{e*}_S\To\Aug _O^*$. 
\end{Cor} 

\ \

\subsection{Group schemes $\Spec A$, $(A, I_A)\in\Aug _O(\C M)$} 
\label{S1.7}

Suppose $\C M\in\MF ^e_S$ and $\C A=(A,I_A)\in \Aug _O(\C M)$ is 
given via a special choice of vectors $\bar m^0$, $\bar m^1$ and 
matrices $U\in M_u(S)$, $C\in M_u(O)$ 
under assumptions $C1$ and $C2$ 
from Subsection \ref{S1.4}. 

We can 
describe the structures of $\C M\oplus\C M$ and $A\otimes _OA$ via  
the $S$-basis $\bar m^0\oplus\{0\}$, 
$\{0\}\oplus\bar m^0$ for $M^0\oplus M^0$, 
the $S$-basis $\bar m^1\oplus\{0\}, \{0\}\oplus\bar m^1$ for 
$M^1\oplus M^1$ and the corresponding matrices 
$\left (\begin{array}{cc}U & 0 \\ 0 & U \end{array}\right )\in M_{2u}(S)$ 
and $\left (\begin{array}{cc}C & 0 \\ 0 & C \end{array}\right )\in M_{2u}(O)$. 
(One can easily 
see that these data satisfy assumptions $C1$ and $C2$ from Subsection \ref{S1.4}.)  
Note that $A=O[\bar X]/\C I_A$, where 
the ideal $\C I_A$ is generated by the coordinates 
of the vector $((\bar XC)^{(2)}+2\bar X){C^{(2)}}^{-1}$,  
$$\C A\otimes \C A=(A\otimes _OA, I_A\otimes _OA+A\otimes _OI_A)
\in\Aug _O(\C M)$$ 
and 
$$A\otimes _OA=O[\bar X\otimes 1, 1\otimes\bar X]
/(\C I_A\otimes 1, 1\otimes \C I_A).$$ 

Let $e:A\To A/I_A=O$ be the natural projection and 
let $\Delta ^*_{\c A}=\Psi ^*(\bigtriangledown)\in\Hom 
_{\Aug ^*_O}(\C A,\C A\otimes \C A)$, where  
$\bigtriangledown :\C M\To \C M\oplus\C M$ 
is the class of the diagonal morphism in the category $\MF ^{e*}_S$.

Let $\C L_{\C A}$ be the set of all 
$\Delta \in\Hom _{\Aug _O}(\C A,\C A\otimes \C A)$ such that:
\medskip 

$\bullet $\  $\Delta\operatorname
{mod}R=\Delta ^*_{\C A}$;
\medskip 

$\bullet $\  $G=\Spec A$ becomes an object of the category $\Gr _O$ when provided with 
the counit $e$ and the coaddition $\Delta $. 
\medskip 

\begin{Prop} \label{P1.10} 
 {\rm a)} $\C L_{\C A}\ne\emptyset $; 

{\rm b)} If $\Delta _1,\Delta _2\in\C L_{\C A}$ then the corresponding 
coalgebra structures on $\C A$ are transformed one into another 
via an automorphism $f\in\Aug _O(\C A)$ such that 
$f\underset{R}\sim \id _{\C A}$ (i.e. $f$ and $\id _{\C A}$ coincide in 
$\Aug _O^*$). 
\end{Prop}

\begin{proof} Let $\bar X=(\bar X^{loc}, \bar X^{et})$ with respect 
to the presentation $\bar m=(\bar m^{loc}, \bar m^{et})$ from 
condition $C2$ in Subsection \ref{S1.4}.  
For $\Delta\in\C L_{\C A}$, set 
$$\Delta (\bar X)=(\Delta (\bar X^{loc}), \Delta (\bar X^{et}))=
\bar X\otimes 1+1\otimes \bar X+\bar\jmath ,$$
where $\bar\jmath =(\bar\jmath ^{loc}, \bar\jmath ^{et})$. 
 
Note that $\bar\jmath ^{et}$ does not depend on a choice of 
$\Delta \in\C L_{\C A}$. 
This implies that $G^{et}:=\Spec A^{et}\in\Gr _O^{et}$ when provided with 
the coaddition $\Delta ^{et}=\Delta |_{A^{et}}$ and the counit 
$e^{et}=e|_{A^{et}}$. More explicitly, 
$A^{et}=O[\bar X^{et}]$ 
with the equations $\eta \bar X^{et(2)}\equiv\bar X^{et}\operatorname{mod}2I_{A^{et}}$, 
where $\eta =-\pi ^{2e}/2\in O^*$, and $\Delta (\bar X^{et})
\equiv \bar X^{et}\otimes 1+1\otimes\bar X^{et}
\operatorname{mod}2I_{A^{et}\otimes A^{et}}$.

Let $G_k^{et}:=G^{et}\otimes _Ok=\Spec A^{et}_k$. 
Remind that the $k$-module of symmetric Hochschild 2-cocycles 
$Z^2_{sym}(G_k^{et})$ consists of symmetric 
$\gamma\in I_{A_k^{et}\otimes A_k^{et}} $ such that 
$$(\Delta\otimes \id )\gamma +\gamma\otimes 1=1\otimes\gamma  
+(\id\otimes \Delta )\gamma .$$ 
The corresponding $k$-module of 2-coboundaries equals 
$$B^2(G_k^{et})=\{\Delta ^{et}(\gamma )-\gamma\otimes 1-1\otimes\gamma 
\ |\ \gamma\in I_{A_k^{et}}\}\subset Z^2_{sym}(G_k^{et})$$ 

We have the following two facts:

$\bullet $\  Suppose $\gamma\in Z^2_{sym}(G_k^{et})$ and 
$\mult :A^{et}_k\otimes A^{et}_k\To A^{et}_k$ is 
the morphism of multiplication. Then 
$\gamma\in B^2(G_k^{et})\ \Leftrightarrow\ \ \mult (\gamma )=0$. 

$\bullet $\  If $\bar X^{et}=(X_1^{et},\dots ,X_{u_{et}}^{et})$ then the elements 
$\delta ^+(X^{et}_{i_1}\dots X_{i_s}^{et})\operatorname{mod}\pi $, 
where $s\geqslant 2$, 
$1\leqslant i_1<\dots <i_s\leqslant u_{et}$ and 
$\delta ^+=\Delta -\id \otimes 1-1\otimes \id $, form a $k$-basis 
of $B^2(G_k^{et})$ and the correspondences 
$\delta ^+(X^{et}_{i_1}\dots X_{i_s}^{et})\mapsto X^{et}_{i_1}
\dots X^{et}_{i_s}$ 
determine the $k$-linear embedding $\omega :B^2(G_k^{et})
\To I_A^{et}\operatorname{mod}\pi $. Note that for any 
$\alpha\in B^2(G_k^{et})$, 
$\alpha ^2\in B^2(G_k^{et})$ and $\omega (\alpha ^2)=\omega (\alpha )^2$.

Now notice that $\Delta $ depends only on the residue 
$$\bar \jmath \operatorname{mod}\widetilde{J}_{A\otimes A}
\in J^u_{A\otimes A}\operatorname{mod}\widetilde J_{A\otimes A}
=2I^u_{A^{et}\otimes A^{et}}\operatorname{mod}(2\pi I_{A^{et}\otimes A^{et}})$$
Let $\bar\jmath =2\bar\alpha \operatorname{mod}2\pi I_{A^{et}\otimes A^{et}}$, where 
$\bar\alpha (\Delta )=(\bar\alpha ^{loc}(\Delta ), 
\bar\alpha ^{et})\in I^u_{A^{et}\otimes A^{et}}$. 

We have the following properties:

a) $\Delta $ defines a morphism of augmented algebras iff 
$(\bar\alpha C)^{(2)}+\bar\alpha $ has all its coordinates in 
$\pi I_{A^{et}\otimes A^{et}}$.

b) $\Delta $ determines a structure of commutative group scheme on $G=\Spec A$ 
iff $\bar\alpha \operatorname{mod}\pi I_{A^{et}\otimes A^{et}}$ has all 
its coordinates in $Z^2_{sym}(G_k^{et})$.

c) $2\id _G=0$ iff $\bar\alpha ^{loc}\operatorname{mod}\pi 
I_{A^{et}\otimes A^{et}}$ 
has all its coordinates in $B^2(G_k^{et})$. 

\ \ 

The proof of property a) uses the equations 
$(-1/2)(\bar XC)^{(2)}= \bar X$ for $A$, property 
b) is equivalent to the axioms of 
coassociativity and cocommutativity for $G$. 
As for property c), note that 
$\bar 0=2\id _G(\bar X)=\mult (\Delta (\bar X))=
2\bar X+\mult(\bar\jmath )$ and, therefore, 
$\mult\bar\jmath ^{loc}$ has all its coordinates 
in $\widetilde{J}_{A\otimes A}$ or, equivalently, 
$\mult (\bar\alpha ^{loc})\equiv 0\operatorname{mod}\pi I_{A^{et}}$. 

Now we can finish the proof of Proposition \ref{P1.10}. 

Let $\Delta _0\in\Delta ^*_{\C A}$ be such that 
$$\Delta _0(\bar X^{loc})\equiv \bar X^{loc}
\otimes 1+1\otimes\bar X^{loc}\operatorname{mod}
\widetilde{J}_{A\otimes A}.$$
This means that $\bar\alpha (\Delta _0)=(\bar 0, \bar\alpha ^{et})$ and 
by above properties a)-c), we have $\Delta _0\in \C L_{\C A}$. 

Suppose $\Delta\in\C L_{\C A}$ and $\bar\alpha (\Delta )=(\bar\alpha ^{loc}
(\Delta ),\bar\alpha ^{et})$. Let $\bar\gamma =(\gamma ^{loc},
\bar 0)\in I_{A^{et}}^n$ 
be such that $\delta ^+(\bar\gamma ^{loc}
\operatorname{mod}\pi I_{A^{et}})
=\bar\alpha ^{loc}\operatorname{mod}\pi I_{A^{et}\otimes A^{et}}$. 
We can assume that 
$\bar\gamma ^{loc}=\omega (\delta ^+\bar\alpha ^{loc})$ and, 
therefore, $(\bar\gamma C)^{(2)}+\bar\gamma $ has all its coordinates in 
$\pi I_{A^{et}}$. Therefore, 
there is  a unique $F\in\Hom _{\Aug _O}(\C A, \C A\otimes\C A)$ such that 
$F(\bar X)\equiv\bar X+2\bar\gamma \operatorname{mod}\widetilde{J}_{A\otimes A}$. 
Clearly, $F\underset{R}\sim \id _{\C A}$ in $\Aug ^*_O$. 

In addition, 
$\Delta _0(\bar F(\bar X))\equiv\Delta _0(\bar X+2\bar\gamma )\equiv $
$$(\bar X^{loc}\otimes 1+1\otimes\bar X^{loc}
+2\Delta ^{et}(\bar\gamma ^{loc}), \Delta _0(\bar X^{et}))$$
$$\equiv (\bar X^{loc}\otimes 1+1\otimes\bar X^{loc}+2\bar\jmath ^{loc}
+2(\bar\gamma ^{loc}\otimes 1+1\otimes\bar\gamma ^{loc}), 
\Delta (\bar X^{et}))$$
$$\equiv (F\otimes F)(\Delta (\bar X))\operatorname{mod}
\widetilde{J}_{A\otimes A}.$$

Therefore, $F\circ\Delta _0=\Delta\circ (F\otimes F)$. 
The Proposition is proved.
\end{proof}

\subsection {Functor $\C G_O$.} \label{S1.8} 

\begin{Prop} 
\label{P1.11} 
There is a functor $\C G_O:\MF ^{e*}_S\To\Gr ^*_O$ such that 
 its compositoon with the forgetful functor 
$\Gr ^*_O\To\Aug ^*_O$ coincides 
with the functor $\Psi ^*$ from Corollary {\it \ref{C1.9}}.
\end{Prop}

\begin{proof} For $i=1,2$, let 

$\bullet $\ $\C M_i\in\MF ^e_S$ with specially chosen 
vectors $\bar m_i^0$, $\bar m_i^1$ satisfying the conditions $C1$, $C2$ 
fom Subsection \ref{S1.4};

$\bullet $\ $\C A_i=(A_i, I_{A_i})$ be the corresponding augmented $O$-algebras 
with the coalgebra structures uniquely given by the coadditions 
$\Delta _i:A_i\To A_i\otimes A_i$ such that $\Delta _i\in\Psi ^*(\bigtriangledown _i)$ 
(where $\bigtriangledown _i$ are the diagonal maps from $\C M_i$ to 
$\C M_i\oplus\C M_i$) and 
$\Delta _i(\bar X_i^{loc})
\equiv\bar X_i^{loc}\otimes 1+1\otimes\bar X_i^{loc}
\operatorname{mod}\widetilde J_{A_i\otimes A_i}$. 

Denote the corresponding group schemes $\Spec A_i=G_i\in\Gr _O$. 

Suppose $f\in\Hom _{\CMF _S}(\iota (\C M_1), \iota (\C M_2))=
\Hom _{\MF ^{e*}_S}(\C M_1,\C M_2)$ 
and $F\in \Hom _{\Aug _O}(\C A_1,\C A_2)$ is such that $F\in\Psi ^*(f)$. 
Then 
$$(F\otimes F)(\Delta _1(\bar X_1))-\Delta _2(F(\bar X_1))
\in J_{A_2\otimes A_2}.$$

Let $\bar\alpha =(\bar\alpha ^{loc}, \bar\alpha ^{et})$ be the vector 
with the coordinates in $I_{A_2^{et}\otimes A_2^{et}}$ such that 
$$(F\otimes F)(\Delta _1(\bar X_1))-\Delta _2(F(\bar X_1))
\equiv 2\bar\alpha\operatorname{mod} \widetilde{J}_{A_2\otimes A_2}.$$

Note that if $F^{et}:=F|_{A_1^{et}}$ then 
the congruence $(F^{et}\otimes F^{et})(\Delta _1(\bar X_1^{et}))
\equiv\Delta _2(F^{et}(\bar X_1^{et}))\operatorname{mod}2A_2^{et}$ 
implies that $\bar\alpha ^{et}=\bar 0$ and $F^{et}$ induces a morphism 
of etale group schemes $G_2^{et}\To G_1^{et}$. 

Using that $F$ is a morphism of augmented $O$-algebras we obtain that  
$(\bar\alpha C)^{(2)}+\bar\alpha $ has all  
coordinates in $\pi I_{A^{et}\otimes A^{et}}$.

One can verify easily that $\bar\alpha 
\operatorname{mod}\pi I_{A_2^{et}\otimes A_2^{et}}$ 
has all coordinates in $Z^2_{sym}(G_2^{et}\otimes k)$ and using that for $i=1,2$, 
$2\,\id _{G_i}=\Delta _{G_i}\circ \mult $ we obtain that 
$\bar\alpha \operatorname{mod}\pi I_{A_2^{et}\otimes A_2^{et}}$ has all 
coordinates in $B^2(G_2^{et}\otimes k)$. 

Let $\bar\gamma =(\gamma ^{loc},\bar 0)$ be the vector with 
coordinates in $I_{A_2^{et}}$ such that 
$\bar\gamma\operatorname{mod}\pi I_{A_2^{et}}=
\omega (\bar\alpha \operatorname{mod}\pi I_{A_2^{et}\otimes A_2^{et}})$. 
Then 
$$\Delta _2^{et}(\bar\gamma )\equiv 
\bar\gamma \otimes 1+1\otimes\bar\gamma +\bar\alpha \operatorname{mod}
\pi I_{A_2^{et}\otimes A_2^{et}}$$ 
and 
$(\bar\gamma C)^{(2)}+\bar\gamma $ has all coordinates in $\pi I_{A_2^{et}}$. 
This implies that there is an $F'\in\Aug _O(\C A_1,\C A_2)$ such that 
$F'(\bar X_1^{et})=F(\bar X_1^{et})$ and 
$F'(\bar X^{loc})\equiv F(\bar X_1^{loc})+2\bar\gamma \operatorname{mod}
\widetilde{J}_{A_2\otimes A_2}$. Therefore, 
$$(F'\otimes F')(\Delta _1(\bar X_1))\equiv \Delta _2(F'(\bar X_1))
\operatorname{mod}\widetilde{J}_{A_2\otimes A_2},$$
and $(F'\otimes F')\circ\Delta _1=\Delta _2\circ F'$. 
This proves the existence of $F'\in\Psi ^*(f)$ such that 
$\Spec F'\in\Hom _{\Gr _O}(G_2, G_1)$. 

Similarly, one can verify that if $F''\in\Psi ^*(f)$ is such that 
$\Spec (F'')\in\Hom _{\Gr _O}(G_2,G_1)$ then $F'$ and $F''$ are equivalent 
in the category $\Gr _O^*$ with the obvious inverse statement. 
The proposition is proved.
\end{proof}

\subsection{Full faithfulness of $\C G_O$} 
\label{S1.9} 

Suppose $\C N=(N^0,N^1,\varphi _1)\in\MF ^e_S$ 
and $\C G_O(\C N)=\Spec B$. 

\subsubsection{Special construction of $B$.} \label{S1.9.1} 

Use the following special case of the construction of 
the $O$-algebra $B$ from Subsection \ref{S1.4}. 

Let $n_1^1,\dots ,n^1_u$ be an $S$-basis of $N^1$ such that 
there are $\tilde s_1,\dots ,\tilde s_u\in S$ and an $S$-basis 
$n_1,\dots ,n_u$ of $N^0$ such that 
for $1\leqslant i\leqslant u$, $n_i^1=\tilde s_in_i$. One 
can easily see that 
this choice of $\bar n^1=(n_1^1,\dots ,n^1_u)$ 
can be made in such a way that conditions $C1$ and $C2$ 
from Subsection \ref{S1.4}. 
are satisfied. Also notice that all $\tilde s_i$ divide $t^e$. 

Set $\bar n^0=\varphi _1(\bar n^1)$ and let the matrices 
$U\in M_u(S)$ and $U_0\in\GL _u(S)$ be such that 
$\bar n^1=\bar n^0U$ and $\bar n^0U_0=\bar n:=(n_1,\dots ,n_u)$. 
Then $U=U_0U_1$, where $U_1=(\tilde s_i
\delta _{ij})_{1\leqslant i,j\leqslant u}$ 
is diagonal. Choose  $\widetilde{\Omega}=
(\tilde\eta '_i\delta _{ij})_{1\leqslant i,j\leqslant u}
\in M_u(O)$ and $D\in \GL _u(O)$ such that 
$\kappa _{SO}(U_1\operatorname{mod}t^{2e})=\widetilde{\Omega }
\operatorname{mod}2$ and 
$\kappa _{SO}(U_0\operatorname{mod}t^{2e})=D\operatorname{mod}2$. 
Then for $\bar Y=(Y_1,\dots ,Y_u)$, the $O$-algebra $B$ is 
the quotient of $O[\bar Y]:=O[Y_1,\dots ,Y_u]$ by the 
ideal generated by the coordinates of 
the vector 
$$\left ((\bar YD\widetilde{\Omega })^{(2)}+2\bar Y\right )
{(D\widetilde{\Omega })^{(2)}}^{-1}$$
Then in the new coordinates $\bar X=(X_1,\dots ,X_u):=\bar YD$, 
$B$ is the quotient of $O[\bar X]:=O[X_1,\dots ,X_u]$  
by the ideal generated by the elements  
$$X_i^2-\eta _i\sum _jX_jc_{ji},\ \ 1\leqslant i\leqslant u,$$
Here $C=(c_{ij})=D^{-1}$,  and for all $i$, 
$\eta _i=-2/\tilde\eta ^{\prime 2}$. With this notation, 
the counit $e:B\To O$ and the coaddition 
$\Delta :B\To B\otimes _OB$ are uniquely recovered 
(in the category $\Gr ^*_O$) from the conditions 
$e(X_i)=0$ and $\Delta (X_i)\equiv 
X_i\otimes 1+1\otimes X_i\operatorname{mod}J_{B\otimes B}$. 

Remind that $\bar X=(\bar X^{loc}, \bar X^{et})$, 
where for $u_0=\dim _k\bar N^{0,loc}$, 
$\bar X^{loc}=(X_1,\dots ,X_{u_0})$ and 
$\bar X^{et}=(X_{u_0+1},\dots ,X_u)$. Then condition $C2$ implies:

$\bullet $\ for $1\leqslant i\leqslant u_0$, 
$X_i, X_i^2/\eta _i\in I_B^{loc}$;

$\bullet $\ for $u_0<i\leqslant u$, $\eta _i\in O^*$ and 
$X_i\in I_{B^{et}}$.

Therefore, the matrix $C=(c_{ij})$ has the following block structure 
$C=\left (\begin{array} {c c}C_0 & 0 \\ C' & C^{et} \end{array} \right )$, 
where $C_0\in\GL _{u_0}(O)$, $C^{et}\in\GL _{u^{et}}(O)$ with $u^{et}:=u-u_0$,  
$C'\equiv 0\operatorname{mod}\pi $. In particular, 

${\bf (C)} $ {\it if $C=(c_{ij})$ and $D=C^{-1}=(d_{ij})$ then 
$c_{ij}\equiv d_{ij}\equiv 0\operatorname{mod}\pi $ if either 
$i\leqslant u_0<j$ or $j\leqslant u_0<i$. }
\medskip 

Let $\delta ^+X_i=j_i\in I_{B\otimes B}$, $1\leqslant i\leqslant u$. 
The coordinates  
of the vector $\bar\jmath =(\bar\jmath ^{loc},\bar\jmath ^{et})$, 
where $\bar\jmath ^{loc}=(j_1,\dots ,j_{u_0})$ and 
$\bar\jmath ^{et}=(j_{u_0+1},\dots ,j_u)$, appear 
as the solutions in $I_{B\otimes B}$ of the 
system of equations  
\begin{equation} \label{E1.1}
 \sum _sj_sc_{si}=-\tilde\eta _iX_i\otimes X_i-(\tilde\eta _iX_1\otimes 1
+1\otimes \tilde\eta _iX_i)j_i-\tilde\eta _ij_i^2/2, 
\end{equation}
where $1\leqslant i\leqslant u$ and 
$\tilde\eta _i=-2/\eta _i=\tilde\eta ^{\prime 2}$. 

The coordinates of $\bar\jmath ^{et}$ are determined by these equations 
uniquely and belong to the ideal $J_B$. The coordinates of 
$\bar\jmath ^{loc}$ are unique under the assumption that 
$j_1,\dots ,j_{u_0}\in \widetilde{J}_{B\otimes B}\subset J_{B\otimes B}$. 
\medskip 

\begin{remark} a) One can easily verify that 
the above system of equations when considered 
modulo any $DP$-ideal $\widetilde{I}$ 
of $B$ such that $\widetilde{I}\subset\widetilde{J}_{B\otimes B}$ 
has a unique solution $\bar\jmath\operatorname{mod}\widetilde{I}$ 
under the assumption that for $1\leqslant i\leqslant u_0$, all 
$j_i\in\widetilde{J}_{B\otimes B}$. 

b) The above description of the coalgebra $B$ 
is related to a very special choice of 
$S$-bases in $N^0$ and $N^1$. This choice is sufficient for   
the formal construction of the algebras $\c A(\c M)$ in Subsection \ref{S1.3}.   
But 
when proving the full faithfulness of $\c G_O$ below in Section \ref{S3} we need 
a choice of appropriate bases which is 
compatible with extensions in $\MF ^e_S$. 
Such choice is possible 
under more general assumptions from Subsection \ref{S1.3}.   
\end{remark}

\subsubsection{Recovering $\C N$} \label{S1.9.3}

In the above construction of the $O$-algebra $B$, any 
$a\in I_B$ can be uniquely written as 
\begin{equation} 
\label{E1.2} 
a=\sum _{ \underline{i}} o_{\underline{i}}
X^{\underline{i}}=\sum _{\underline{i}}  
o_{\underline{i}}X_1^{i_1}\dots X_u^{i_u}
\end{equation} 
where (by our general agreement from Introduction) 
$\underline{i}=(i_1,\dots ,i_u)$ is a non-zero vector with the coordinates 
$i_1,\dots ,i_u\in\{0,1\}$, and all coefficients 
$o_{\underline{i}}=o_{\underline{i}}(a)$ belong to $O$. 
Similarly, any $a\in I_{B\otimes B}$ can be uniquely written as 
an $O$-linear combination of the elements $X^{\underline{i}_1}\otimes 1$, 
$1\otimes X^{\underline{i}_2}$ and 
$X^{\underline{i}_1}\otimes X^{\underline{i}_2}$

Consider the following property of ideals $I\subset I_B$ in $B$ 
(or with the obvious changes in $B\otimes B$, 
$\bar B$, $\bar B\otimes\bar B$, etc.). 

\begin{equation} \label{AE3} 
 \sum _{\underline{i}}o_{\underline{i}}X^{\underline{i}}\in I\ \Leftrightarrow 
\ \ \forall\underline{i},\ \ o_{\underline{i}}X^{\underline{i}}\in I
\end{equation}

\begin{remark}
 Suppose $I_1$ and $I_2$ satisfy property \eqref{AE3}. Then 

a) $I_1+I_2$ satisfies property \eqref{AE3};

b) for any monomial $o_{\underline{i}}X^{\underline{i}}$, 
$o_{\underline{i}}X^{\underline{i}}\in I_1+I_2$ iff 
either $o_{\underline{i}}X^{\underline{i}}\in I_1$ or 
$o_{\underline{i}}X^{\underline{i}}\in I_2$.
\end{remark}

\begin{Prop} \label{PA1}
The ideals $J_B$ and $\widetilde{J}_B$ satisfy property {\rm \eqref{AE3}}.
\end{Prop}

\begin{proof} 
Any element of $J_B$ is a sum of ``elementary'' elements  of the form 
$o_{\underline{i}_1}X^{\underline{i}_1}o_{\underline{i}_2}X^{\underline{i}_2}$ 
and $\pi ^eo_{\underline{i}_1}X^{\underline{i}_1}$, where 
$o_{\underline{i}_1}X^{\underline{i}_1}, 
o_{\underline{i}_2}X^{\underline{i}_2}\in I_B(2)$. It will be 
sufficient to verify property (\ref{AE3}) for such elementary elements. 

If $a=\pi ^eo_{\underline{i}_1}X^{\underline{i}_1}$ there is nothing to prove. 

Suppose $a=o_{\underline{i}_1}X^{\underline{i}_1}o_{\underline{i}_2}X^{\underline{i}_2}$ 
and $\underline{i}_1=(i_{11},\dots ,i_{u1})$, 
$\underline{i}_2=(i_{12},\dots ,i_{u2})$. Use induction on the number 
$\nu =\nu (a)$ of $1\leqslant j\leqslant u$ such that 
$i_{j1}=i_{j2}=1$. 

If $\nu (a)=0$ then $\underline{i}_1+\underline{i}_2=\underline{i}$ 
and there is nothing to prove. If $\nu (a)\geqslant 1$ and, say, $i_{j1}=i_{j2}=1$ 
use the identity $X_j^2=\eta '_j\cdot \eta _j'\sum _{s}X_sc_{sj}$ 
to rewrite $a$ as a sum of elements with smaller $\nu $'s and, perhaps, 
elements of the form $\pi ^eI_B(2)$. 

The case of the ideal $\widetilde{J}_B$ can be considered similarly. 
\end{proof}

\begin{Prop} \label{P1.13} 
 If $\iota _{\C N}(\C N)=
(\widetilde{N}^0, \widetilde{N}^1,\varphi _1)$ then 
$$\widetilde{N}^0=\{a\operatorname{mod}J_B\ |
\ a\in I_B, \delta ^+(a)\in J_{B\otimes B}\}$$
\end{Prop} 

\begin{proof} Suppose 
$a=\sum _{\underline{i}}o_{\underline{i}}X^{\underline{i}}\in I_B$ 
and $\delta ^+a\in J_{B\otimes B}$. Note that    
$$\delta ^+(a)\equiv \sum _{\underline{i}_1+\underline{i}_2=\underline{i}}
o_{\underline{i}}X^{\underline{i}_1}\otimes X^{\underline{i}_2}
\operatorname{mod}J_{B\otimes B}$$
Then Proposition \ref{PA1} implies that all 
$o_{\underline{i}}X^{\underline{i}_1}\otimes X^{\underline{i}_2}\in J_{B\otimes B}$ 
and, therefore, all $o_{\underline{i}}
X^{\underline{i}_1}\cdot X^{\underline{i}_2}\in J_B$. 
This means that all non-linear terms amongst $o_{\underline{i}}X^{\underline{i}}$ 
(i.e. the terms with $r(\underline{i})=i_1+\dots +i_u\geqslant 2$) 
belong to $J_B$. 
\end{proof}

Using that the ideals $J_B$ and $J_{B\otimes B}$ depend functorially on 
the group scheme $\c G_O(\c N)$ (i.e. do not depend on a choice of the 
special construction in Subsection \ref{S1.9}) we obtain the following property.

\begin{Cor} 
\label{C1.13} 
The functor $\C G_O$ is fully faithful. 
\end{Cor}

\begin{proof}
 Suppose $G_1=\c G_O(\c M_1)$, 
$G_2=\c G_O(\c M_2)$, $\c M_1,\c M_2\in \MF ^e_S$, 
$g\in \Hom _{\Gr _O}(G_1,G_2)$ and  
$A(g):A_2\To A_1$ is the corresponding morphism of  
$O$-algebras. Then  $\iota (A(g))
\in\Hom _{\CMF _S}(\iota (A_2),\iota (A_1))$ maps 
$\iota _{\c M_2}(\c M_2)$ to $\iota _{\c M_1}(\c M_1)$ 
(use Proposition \ref{P1.13}) and by Proposition \ref{P1.1} 
can be lifted uniquely $F\in\Hom _{\MF ^e_S}(\c M_2,\c M_1)$. 
Clearly, $\c G_O(F)=g$.  
\end{proof}

In Subsection \ref{S3} we need the following version of 
Proposition \ref{P1.13}.  

Let $\bar B=B\otimes_O\bar O$.  
Denote by $\bar{\C J}$ the ideal in $\bar B$ generated by 
$\tilde\eta _iX_i\otimes X_i$, $1\leqslant i\leqslant u_0$,  
and all elements of $2I_{\bar B\otimes\bar B}^{loc}$. One can easily 
prove (use relation \eqref{E1.1}) that all $j_1,\dots ,j_{u_0}$ belong to 
$\bar{\c J}$ and $\pi j_{u_0+1},\dots ,\pi j_u
\in 2\pi I_{B\otimes B}\subset\bar{\c J}$. 

\begin{Prop} \label{P1.14} 
The ideal $\bar{\c J}$ consists of all $\bar O$-linear combinations 
of monomials in $\bar B\otimes\bar B$ which either belong to 
$2I_{\bar B\otimes\bar B}^{loc}$ or are divisible by one of 
$\tilde\eta _iX_i\otimes X_i$, where $1\leqslant i\leqslant u_0$.  
\end{Prop}

\begin{proof} Suppose an element $a\in\bar{\c J}$ 
is an $\bar O$-linear combination of the form 
\begin{equation}\label{EE}
 a=\sum _{\underline{i}}o'_{\underline{i}}
(X^{\underline{i}}\otimes 1)+\sum _{\underline{i}}
o_{\underline{i}}''(1\otimes X^{\underline{i}})+
\sum _{\underline{i}_1,\underline{i}_2}o_{\underline{i}_1\underline{i}_2}
X^{\underline{i}_1}\otimes X^{\underline{i}_2}
\end{equation}
Let $M(i,b)=\tilde\eta _i(X_i\otimes X_i)b$, where $1\leqslant i\leqslant u_0$ 
and $b$ is a monomial from $\bar B\otimes\bar B$. 

Clearly, any $M(i,b)$ can't contribute to the coefficients of the first two 
sums in \eqref{EE}. Therefore, their summands belong 
to $2I_{\bar B\otimes\bar B}^{loc}$. 

Suppose $o_{\underline{i}_1\underline{i}_2}
X^{\underline{i}_1}\otimes X^{\underline{i}_2}
\notin 2I_{\bar B\otimes\bar B}^{loc}$ and satisfies the following condition
(The proposition is proved if there are no such monomials.): 
\medskip 

{\it if $X^{\underline{i}_1}\otimes X^{\underline{i}_2}$ is divisible by 
$X_i\otimes X_i$, $1\leqslant i\leqslant u_0$, then 
$o_{\underline{i}_1\underline{i}_2}$ is not divisible by $\tilde\eta _i$.}
\medskip 

Then there is $M(i_0,b)$ with $1\leqslant i_0\leqslant u_0$, 
which contributes to the coefficient 
for $X^{\underline{i}_1}\otimes X^{\underline{i}_2}$ and this contribution 
divides $o_{\underline{i}_1\underline{i}_2}$. 

If $b$ is not divisible by either 
$X_{i_0}\otimes 1$ or $1\otimes X_{i_0}$ then 
$M(i_0,b)=oX^{\underline{i}_1}\otimes X^{\underline{i}_2}$ and $o\in\bar O$ 
is divisible by 
$\tilde\eta _{i_0}$. The contradiction. 
But otherwise, $M(i_0,b)\in 2I_B^{loc}$ 
because $\tilde\eta _{i_0}X_{i_0}^2\in 2I_B^{loc}$. 
The proposition is proved. 
 \end{proof}

\begin{Cor}\label{CC}
The ideal $\bar{\c J}$ satisfies condition \eqref{AE3}. 
\end{Cor}

\medskip 
\medskip

\section {Functor $\C G_{O_0}^O$} \label{S2} 

In this section we prove that any $G\in\Gr _O$ from the image 
of $\C G_O$ has a canonical descent to $O_0$, $G_0\in\Gr _{O_0}$. 
Therefore, the fully faithful functor $\C G_O$ appears as the composition of 
the fully faithful functor $\C G_{O_0}^O:\MF^{e*}_S\To\Gr ^*_{O_0}$ 
and the extension of scalars $\Gr _{O_0}^*\To \Gr ^*_O$.

\subsection{Uniqueness of descent to $O_0$} \label{S2.1} 

\begin{Prop} \label{P2.1} 
{\rm a)}  Suppose $G=\Spec A\in\operatorname{Im}\C G_O$, 
 $e:A\To O$ is the counit and $I_A=\Ker\,e$. If there is an 
$(A_0, I_{A_0})\in\Aug _{O_0}$ such that $(A_0, I_{A_0})
\otimes _{O_0}O=(A,I_A)$ then 
$G_0=\Spec A_0$ has a natural structure of object of the category 
$\Gr _{O_0}$ such that $G_0\otimes _{O_0}O=G$.  

{\rm b)} Suppose $G_0, H_0\in\Gr _{O_0}$ and 
$G=G_0\otimes _{O_0}O$ and $H=H_0\otimes _{O_0}O$ 
are in the image of $\C G_O$. Then the natural map 
$f\mapsto f\otimes _{O_0}O$ induces identification 
$\Hom _{\Gr _{O_0}}(G_0, H_0)=\Hom _{\Gr _O}(G, H)$.
\end{Prop}

\begin{proof} It will be sufficient to prove that $\Delta (I_{A_0})\subset 
 I_{A_0\otimes A_0}$, where $\Delta $ is the coaddition on $A$. 

Let $\Gal (K/K_0)=\{\id , \tau \}$ and $\Delta ^{(\tau )}
=(\tau\otimes \tau )\circ\Delta\circ \tau $ is the conjugate to $\Delta $. 
In other words, if $b_1,\dots ,b_u$ is an $O_0$-basis of $I_{A_0}$ 
and for $1\le i\le u$, $\Delta (b_i)=b_i\otimes 1+1\otimes b_i+\sum _{k,l}
\alpha _{kl}b_k\otimes b_l$ with all $\alpha _{kl}\in O$, then 
$\Delta ^{(\tau )}(b_i)=b_i\otimes 1+1\otimes b_i+\sum _{k,l}
\tau (\alpha _{kl})b_k\otimes b_l$. Using that all 
$\tau (\alpha _{kl})\equiv \alpha _{kl}\operatorname{mod}2\pi O$, 
we conclude that for any $a\in I_A$, 
$\Delta ^{(\tau )}(a)\equiv\Delta (a)\operatorname{mod}\widetilde{J}_{A\otimes A}$. 
Therefore, by results of Subsection \ref{S1.7}, 
$\Delta =\Delta ^{(\tau )}$ and all $\alpha _{kl}\in O_0$. 
Part a) is proved. 

Part b) follows by similar arguments. 
\end{proof}

\subsection { Existence of descent to $O_0$} 
\label{S2.2} 

\begin{Prop} 
\label{P2.2} 
Suppose $G=\Spec A\in\operatorname{Im}\,\c G_O$ 
and $I_A=\Ker \,e$, where $e:A\To O$ is the counit. 
 Then there is an $(A_0,I_{A_0})\in\Aug _{O_0}$ such that 
$I_A=I_{A_0}\otimes _{O_0}O$. 
\end{Prop}

\begin{proof} Use induction on the order $|G|$ of $G$. 
 
\subsubsection{The case $|G|=2$} \label{S2.2.1} 

Here $A=O[X]$, where $X^2=\eta cX$ with 
$\eta \in O_0$, $\eta |2$ and $c\in O^*$. 
Clearly, we can take $A_0=O_0[c^{-1}X]$. 

\subsubsection{Tame descent} 
\label{S2.2.2} 

Suppose $K_0'$ is a tamely ramified extension of $K_0$ of degree $e_0$. Let 
$\pi _0'$ be a uniformising element of $K_0'$
such that $\pi _0^{\prime e_0}=\pi _0$. 
Let $K'=K_0'(\pi ')$, where $\pi ^{\prime 2}=\pi _0'$. We can assume that 
$\pi ^{\prime e_0}=\pi $. The field extensions 
$K_0'/K_0$ and $K'/K$ are Galois, 
their Galois groups are cyclic of order $e_0$ and can be 
naturally identified. Denote by 
$O_0'$ and $O'$ the valuation rings of $K_0'$ and, resp., $K'$. 

\begin{Lem} 
\label{L2.3} 
Let $G\in \Gr _O$ and $G'=G\otimes _OO'\in\Gr _{O'}$. Then:
\newline 
{\rm a)} $G$ is in the image of $\C G_O$ if and only if $G'$ 
is in the image of $\C G_{O'}$; 
\newline 
{\rm b)} $G$ admits a descent to $O_0$ if and only 
if $G'$ admits a descent to $O_0'$. 
\end{Lem}

\begin{proof} 
The proof is based on an application of the 
criterion of tamely ramified descent. In the case of $O'$-algebras 
this criterion can be stated as follows:

\ \ 

$\bullet $\ {\it Suppose $A'$ is a flat $O'$-algebra and
 $\tau $ is a generator of $\Gal (K'/K)$. Then the 
existence 
of a 
flat $O$-algebra $A$ 
such that $A'=A\otimes _OO'$ is equivalent to the existence of 
a $\tau $-linear automorphism $f$ of $A'$ such that 
$f^{e_0}=\id _{A'}$ and $f\otimes _Ok=\id _{A'\otimes k}$. }

\ \ 

Then one can state a similar criterion for objects of $\MF ^{e_0e}_{S'}$ 
where $S'$ is a tamely ramified extension of $S$ of degree $e_0$ and 
deduce part a) from the fact that tame descent data for $G'=\C G_{O'}(\C M')$ 
induce tame descent data for $\C M'$. Similarly, part b) can be proved 
from the fact that tamely ramified descent data for $G'$ 
induce tamely ramified descent data for $G'_0$ (use the uniqueness of 
$G'_0$ given by Proposition \ref{P2.1}). Cf. for more detailed explanation in 
\cite{Ab}, Proposition 4.3. 
\end{proof}

\subsubsection{Lubin-Tate formal group law} 
\label{S2.2.3}

Consider the formal Lubin-Tate group law with the logarithm 
$$l_{LT}(X)=
X+X^2/2+\dots +X^{2^n}/2^n+\dots \in\Q _2[[X]]$$
 This means that for any $2$-adic 
ring $R$ its topological nilradical $\m (R)$ is provided 
with the structure of abelian group via the operation 
$[f]+[g]=[P(f,g)]$, 
where $f,g\in\m (R)$ and for indeterminants $X,Y$, 
$$P(X,Y)=l_{LT}^{-1}(l_{LT}(X)+l_{LT}(Y))\in\Z _2[[X,Y]].$$

We have the following simple properties:
\medskip  

$\bullet $\ $[f]+[g]=[P_0(f,g)]+\dots +[P_n(f,g)]+\dots $, where all 
$P_n\in\Z _2[X,Y]$ and $\deg P_n=2^n$. In particular, 
$P_0=X+Y$, $P_1=-XY$, $P_2=-XY(X+Y)^2$;
\medskip 

$\bullet $\ $-[f]=[-f]+[-f^2]+\dots +[-f^{2^n}]+\dots $;
\medskip 

$\bullet $\ $[2](f)=[f]+[f]\equiv 
[2f]+[f^2]+[-2f^2]+\dots +[-2f^{2^n}]+\dots \operatorname{mod}4\m (R).$
\medskip 

\subsubsection {The case $|G|>2$} \label{S2.2.4} 

By replacing $O$ by its suitable tamely ramified extension 
we can assume that $G=\C G_O(\C M)$, where 
$\C M=(M^0, M^1,\varphi _1)\in\MF ^e_S$ is such that for $u\geqslant 1$:
\medskip 

$\bullet $\ \ there is an $S$-basis $m^1,n_1^1,
\dots ,n_u^1$ of $M^1$ and an $S$-basis 
$m,n_1,\dots ,n_u$ of $M^0$ such that $\varphi _1(m^1)=m$ and 
for $1\leqslant i\leqslant u$, 
there are $\tilde s_i\in S$, $\tilde s_i|t^e$ such that 
$n_i^1=\tilde s_in_i$, $\varphi _1(n_1^1)=\sum _jn_ju_{ji}$, 
where $(u_{ij})\in\GL _u(S)$;
\medskip 

$\bullet $\ \ there is an $\tilde s\in S$, $\tilde s|t^e$ such that 
$m^1=\tilde sm+\sum _i\alpha _in_i$, where for 
$1\leqslant i\leqslant u$, the coefficients 
$\alpha _i\in S$ and $t^e\tilde s^{-1}\alpha _i
\equiv 0\operatorname{mod}\tilde s_i$.

\ \ 

The above conditions simply mean that 
$\C N=(N^0, N^1,\varphi _1)\in\MF ^e_S$, where 
$N^0=\sum _iSn_i$, $N^1=\sum _iSn^1_i$ and there is a short exact sequence 
in $\MF ^e_S$ 
$$0\To \C N\To \C M\To \C M_{\tilde s}\To 0$$
where $\c M_{\tilde s}=(S\tilde m, S\tilde m^1,\varphi _1)$ with 
$\tilde m^1=\tilde s\tilde m$ and $\varphi _1(\tilde m^1)=\tilde m$. 
We shall always assume that the above data for the structure of 
$\C N\in\MF ^e_S$ satisfy assumptions $C1$ and $C2$ from 
Subsection \ref{S1.4}.

Note that in the above description of 
$\C M\in\Ext _{\MF ^e_S}(\c M_{\tilde s},\C N)$ we can replace 
$m$ by $m'=m+v$ and $m^1$ by $m^{1\prime }=m^1+v^1$, where $v\in N^0$, $v^1\in N^1$ and 
$\varphi _1(v^1)=v$. Then 
$m^{1\prime }=\tilde sm'+\sum _i\alpha '_in_i$, where 
$$\sum _i\alpha '_in_i=\sum _i\alpha _in_i+v^1-\tilde s\varphi _1(v^1).$$
In particular, if $\tilde s\in S^*$ we can always assume that 
\begin{equation} \label{aE1}
\sum _i\alpha _in_i\in N^1+tN^0
\end{equation}  

In terms of the corresponding $O$-algebras we have: 

\ \ 

$\bullet $\ $A=A(G)$ 
contains the $O$-algebra $B=O[X_1,\dots ,X_u]$, where 
 $X_i^2=\eta _i\sum _jX_jc_{ji}$ with 
$\eta _i=-2/\tilde\eta _i^{\prime 2}\in O_0$ and  
$\tilde\eta _i'\operatorname{mod}2=\kappa _{SO}
(\tilde s_i\operatorname{mod}t^{2e})$,  $1\leqslant i\leqslant u$,   
and $C=(c_{ij})\in\GL _u(O)$; 
\medskip 

$\bullet $\ $A=B[Y]$, where $(\tilde\eta 'Y+\sum _ir_iX_i)^2+2Y=0$ or, equivalently,  
\begin{equation} \label{aE2}
 \left (Y+\tilde\eta ^{\prime -1}\sum _ir_iX_i\right )^2-\eta Y=0
\end{equation} 
with $\tilde\eta '\operatorname{mod}2=\kappa _{SO}(\tilde s\operatorname{mod}t^{2e})$, 
$r_i\operatorname{mod}2=\kappa _{SO}(\alpha _i\operatorname{mod}t^{2e})$ 
and 
\begin{equation} \label{aE3}
\eta r_i^2\equiv 0\operatorname{mod}\tilde\eta _i\  
\text {(or, equivalently,}\  \eta _ir_i^2\equiv 0\operatorname{mod}\tilde\eta )
\end{equation}  
for all $1\leqslant i\leqslant u$ (this follows from the above congruences 
$t^e\tilde s^{-1}\alpha _i\equiv 0\operatorname{mod}\tilde s_i$). Recall that 
as usually $\eta\tilde\eta =\eta _i\tilde\eta _i=-2$ for all $i$. 
\medskip 

Let $h=\sum _ir_iX_i\in I_B$. Then (\ref{aE3}) implies that $h^2\in\tilde\eta I_B$. 
In particular, if $\tilde\eta\notin O_0^*$ then $h\in I_B^{loc}$. 
If $\tilde\eta\in O_0^*$ then (\ref{aE1}) implies that again $h\in I_B^{loc}$. 
Therefore, $\tilde\eta 'Y\in I_A^{loc}$.

By inductive assumption, there is an augmented flat $O_0$-algebra 
$(B_0, I_{B_0})$ such that $I_B=I_{B_0}\otimes _{O_0}O$. Therefore, 
$h=b_0+\pi b_1$, where $b_0\in I_{B_0}^{loc}$ and $b_1\in I_{B_0}$. 

From now on we use the Lubin-Tate group law. 
Clearly, there is 
$Y'\in I_A$ such that 
$[\tilde\eta Y']=[\tilde\eta Y+\tilde\eta 'h]-
[\tilde\eta 'b_0]-[\tilde\eta '\pi b_1]$. 
If $\tilde\eta\notin O_0^*$ then 
$Y'\equiv Y\operatorname{mod}\pi I_A$ 
and if $\tilde\eta\in O_0^*$ then 
$Y'\equiv Y\operatorname{mod}(YI_A^{loc}+\pi I_B^{loc})$. 
Therefore, $A=B[Y']$. 

The equation for $Y'$ can be found as follows. From (\ref{aE2}) we obtain that 
$(\tilde\eta Y+\tilde\eta 'h)^2=-2\tilde\eta Y$. Then using the 
properties of the Lubin-Tate group law from Subsection \ref{S2.2.3}
we obtain:
$$[2](\tilde\eta Y')=[2(\tilde\eta Y+\tilde\eta 'h)]+[-2\tilde\eta Y]
-([2](\tilde\eta 'b_0)+[2](\tilde\eta '\pi b_1))\equiv $$
$$-[\tilde\eta b_0^2]-[\tilde\eta\pi _0b_1^2]-\sum _{n\geqslant 0}
\left [-2(\tilde\eta b_0^2)^{2^n}\right ]-\sum _{n\geqslant 0}
\left [-2(\tilde\eta \pi _0b_1^2)^{2^n}\right ]\operatorname{mod}4(\tilde\eta I_A)^{loc}$$
Here $(\tilde\eta I_A)^{loc}$ coincides with $\tilde\eta I_A$ if $\tilde\eta\notin O_0^*$ 
and with $I_A^{loc}$, otherwise. 

Notice that $h^2\in\tilde\eta I_B$ and, therefore, the right 
hand side of the above congruence equals 
$\tilde\eta ^2b^*\in (\tilde\eta ^2I_{B_0})^{loc}$. 
Using that $[2](2(\tilde\eta I_A)^{loc})=4(\tilde\eta I_A)^{loc}$ we can replace 
$Y'$ by $Y_1=Y'-2a$ with a suitable $a\in I_A$ such that 
$[2](\tilde\eta Y_1)=\tilde\eta ^2b^*$. 

Finally, if $1+\tilde\eta Y_2=\exp (l_{LT}(\tilde\eta Y_1))$ then 
we still have $A=B[Y_2]$ and 
$(1+\tilde\eta Y_2)^2=\exp (l_{LT}(\tilde\eta ^2b^*))$ implies that 
$Y_2^2-\eta Y_2=b^*_1$ with $b_1^*\in B_0$. So, 
for $A_0=B_0[Y_2]$, $A_0\otimes _{O_0}O=A$. 
\end{proof} 
\medskip 
\medskip

\section{Surjectivity of $\C G_{O_0}^O$} 
\label{S3}

In order to establish that $\C G^O_{O_0}$ is antiequivalence of the categories 
$\MF ^{e*}_S$ and $\Gr ^*_{O_0}$ it remains only to prove the following result.

\begin{Thm} 
\label{T3.1} 
If $G_0\in\Gr _{O_0}$ then there is  
$\C M\in\MF ^e_S$ such that $\C G_O(\C M)\simeq G_0\otimes _{O_0}O$.
\end{Thm}
\medskip 

 The proof will be given in Subsections 3.1-3.11. 
It uses induction on the order $|G_0|$ of $G_0$. 
\medskip

\subsection {The case $|G_0|=2$} \label{S3.1} 

If $|G_0|=2$ then there is  $\eta\in O_0$, $\eta |2$, such that 
$G_0=\mu _{\eta }$, where 
$\mu _{\eta }=\Spec\, O_0[X]$, $X^2=\eta X$, $e(X)=0$ and 
$\Delta (X)=X\otimes 1+1\otimes X+\tilde\eta X\otimes X$ 
with $\eta\tilde\eta =-2$. We can assume that 
$\tilde\eta =\pi _0^r$, $0\leqslant r\leqslant e$ 
(because $\mu _{\eta }\simeq\mu _{\eta '}$ iff $\eta ^{-1}\eta '\in O_0^*$). 
Then for $\C M=(Sm, St^rm,\varphi _1)\in\MF ^e_S$ 
with $\varphi _1(t^rm)=m$, one has $\C G_O(\C M)=\mu _{\eta }\otimes _{O_0}O$.

\subsection{Basic strategy I} \label{S3.2} 

When studying the case $|G_0|>2$ we can replace 
$O_0$ by its tamely ramified extension, cf. Subsection \ref{S2.2.2}. 
In particular, we can assume 
that in $\Gr _{O_0}$ there is a short exact sequence 
$$0\To \mu _{\eta }\To G_0\To H_0\To 0,$$ 
where $H_0=\Spec B_0$, and $H=H_0\otimes _{O_0}O=\C G_O(\C N)$ with $\C N\in\MF ^e_S$. 

Use the description of $B=B_0\otimes _{O_0}O$ from Subsection \ref{S1.9.1}.  

Namely, $B=O[X_1,\dots ,X_u]$ with the relations 
$X_i^2=\eta _i\sum _jX_jc_{ji}$, 
where all $\eta _i\in O_0$, $\eta _i|2$, 
$C=(c_{ij})\in\GL _u(O)$, $e(X_i)=0$ and 
$j_i=\delta ^+(X_i)=\Delta (X_i)-
X_i\otimes 1-1\otimes X_i\in J_{B\otimes B}$. In addition, if 
$\bar X^{loc}=(X_1,\dots ,X_{u_0})$ then 
$j_1,\dots ,j_{u_0}\in\widetilde{J}_{B\otimes B}$. 
For $1\leqslant i\leqslant u$, the elements 
$\eta _i$ are defined up to units in $O_0$ and we can assume that 
$\tilde\eta _i=\tilde\eta ^{\prime 2}_i=-2/\eta _i$ with $\tilde\eta _i'\in O$. 
Our strategy is to use the explicit construction 
of $H=H_0\otimes _{O_0}O$ from Subsection \ref{S1.9.1}   
to describe $G=G_0\otimes _{O_0}O$ as an element 
of the group $\Ext _{\Gr _O}(H,\mu _{\eta }\otimes _{O_0}O)$. 
If $O'=O[\pi ']$ with $\pi ^{\prime 2}=\pi $, we shall prove 
then that $G'=G\otimes _OO'$ 
appears in the form $\C G_{O'}(\C M')$, where 
$\C M'\in\MF ^{2e}_{S'}$ with $S'=S[t']$, $t^{\prime 2}=t$. 

Finally, we shall prove that the fact that $G'$ admits 
a descent $G_0$ to $O_0$ implies that 
$\C M'$ admits a descent to $S$, i.e. 
$\C M'=\C M\otimes _SS'$ with $\C M\in\MF ^e_S$ and, therefore, $G=\C G_O(\C M)$.

\subsection{The group $\Ext _{\Gr _{O_0}}(H_0,\mu _{\eta })$} 
\label{S3.3} 

Let $B'=B\otimes _OO'$, $I_{B'}=I_B\otimes _OO'$. 
Let $\bar O$ be the valution ring of an algebraic closure 
of $K'=K(\pi ')$, $\bar B=B'\otimes _{O'}
\bar O$ and $I_{\bar B}=I_{B'}\otimes _{O'}\bar O$.  

Introduce the (multiplicative) 
group $\c H=\C H(H_0, \mu _{\eta })$ 
of all elements $f\in (1+\tilde\eta I_{\bar B})^{\times }$ 
such that $f^2\in 1+\tilde\eta ^2I_{B_0}$ and $\delta ^{\times }(f)
:=\Delta (f)(f\otimes f)^{-1}\in 1+\tilde\eta I_{B_0\otimes B_0}$. 

Then there is a group epimorphism 
$\Theta :\C H\To \Ext _{\Gr _{O_0}}(H_0,\mu _{\eta })$ attaching 
to $f\in\C H$, the group scheme $\Theta (f)=\Spec A_0\in\Gr _{O_0}$ such that:

$\bullet $\ $A_0=B_0[X]$ where $(1+\tilde\eta X)^2=f^2$;
\medskip 

$\bullet $\ $e(X)=0$ and $\Delta (1+\tilde\eta X)=
(1+\tilde\eta X)\otimes (1+\tilde\eta X)\cdot \delta ^\times f$;
\medskip 

$\bullet $\  $\Ker\Theta =(1+\tilde\eta I_{B_0})^{\times }$. 
\medskip 

This result was proved in a very detailed way in Section 5 of 
\cite{Ab}   
in a more general context of $p$-group schemes, where $p$ is 
any prime number. Note that the case $p=2$ is 
slightly easier to obtain.

If $\tilde\eta\notin O_0^*$ (i.e. if $\mu _{\eta }$ is not multiplicative) 
then obviously  $f\in 1+I_{\bar B}^{loc}$. 
If $\tilde\eta\in O_0^*$ then we can always multiply $f$ by a suitable element from 
$(1+I_{B^{et}})^{\times }\subset (1+I_B)^{\times}$ 
to assume again that $f\in 1+I_{\bar B}^{loc}$. 
This allows us to replace 
in the above description of $\Ext _{\Gr _{O_0}}(H_0,\mu _{\eta })$, 
 the multiplicative group $(1+I_{\bar B})^{\times }$ 
by $\m (I_{\bar B})=I_{\bar B}^{loc}$ with the 
Lubin-Tate group law from Subsection \ref{S2.2.3}. 

More precisely, let $\delta _{LT}:I_{\bar B}^{loc}
\To I_{\bar B\otimes\bar B}^{loc}$ 
and $[2]:I_{\bar B}^{loc}\To I_{\bar B}^{loc}$ be such that 
for any $f\in I_{\bar B}^{loc}$, 

--- $\delta _{LT}(f)=[\Delta (f)]-[f\otimes 1]-[1\otimes f]$;

--- $[2](f)=[f]+[f]$. 

Let $\c H_{LT}=\c H_{LT}(H_0,\mu _{\eta })$ be the subgroup of $I_{\bar B}^{loc}$ 
(with respect to the Lubin-Tate group law) 
consisting of $f\in I_{\bar B}^{loc}$ 
such that $[2](f)\in\tilde\eta ^2I_{B_0}$ and 
$\delta _{LT}(f)\in\tilde\eta I_{B_0\otimes B_0}$. 
Let $\Theta _{LT}:\C H_{LT}\To \Ext _{\Gr _{O_0}}(H_0,\mu _{\eta })$ be such that 
for all $f\in\C H_{LT}$, 
$\Theta _{LT}(f)=\Theta (E(f))$, where $E(X)=\exp (l_{LT}(X))$ 
is the Artin-Hasse exponential. Our description of 
$\Ext _{\Gr _{O_0}}(H_0,\mu _{\eta })$ will be used below in the 
following form. 

\begin{Prop} 
\label{P3.2} 
$\Theta _{LT}$ is a group epimorphism and its kernel equals 
$\C H_{LT}\cap (\tilde\eta I_{B_0})=(\tilde\eta I_{B_0})^{loc}$.
 \end{Prop}

\subsection{Main Lemma} \label{S3.4} 

In next subsections we work systematically with the Lubin-Tate 
group law from Subsection \ref{S2.2.3}. We always bear in mind 
the following agreement: if, say,  
$a\in I_B$ appears in the form $[a]$ then $a$ is assumed 
automatically to be an element of $I_B^{loc}$, that is 
the corresponding result of the Lubin-Tate addition is 
automatically well-defined. 

Recall that we use the multi-indices 
$\underline{i}=(i_1,\dots ,i_u)$, where all coordinates 
of the vector $\underline{i}$ belong to $\{0,1\}$. 
We shall use such indices for the abbreviation 
$X^{\underline{i}}:=X_1^{i_1}\dots X_u^{i_u}$, especially,  
when $r(\underline{i}):=i_1+\dots +i_u\geqslant 2$. If $r(\underline{i})=1$ 
then the multi-index $\underline{i}$ appears just as 
index $j$, $1\leqslant j\leqslant u$, such that 
$\underline{i}=(\delta _{1j},\dots ,\delta _{uj})$  
(where $\delta $ is the Kronecker symbol). 

The following statement is very similar to the statement 
appeared in Subsection 6.1 of \cite{Ab} as Main Lemma. 

\begin{Lem}[{\bf Main Lemma}] \label{L3.3} 
If $f\in\C H_{LT}=\C H_{LT}(H_0,\mu _{\eta })$ then there are: 

--- $f_0\in\tilde\eta I_B$;

--- for $1\leqslant i\leqslant u$, $o'_{i0}\in O$ and $o'_{i1}\in\pi 'O$;

--- for all multi-indices $\underline{i}$, 
$D_{\underline{i}}\in \bar O$, 

\ \ 

such that $o^{\prime 2}_{i0}, 
o^{\prime 2}_{i1}\in\tilde\eta O$, $D_iX_i\in I_{\bar B}(2)^{loc}$, 
$D_{\underline{i}}X^{\underline{i}}\in I^{loc}_{\bar B}$ and 
$$f=[f_0]+\sum _{1\leqslant i\leqslant u}([o'_{i0}X_i]+[o'_{i1}X_i]+[D_iX_i])
+\left [2\sum _{r(\underline{i})\geqslant 2}D_{\underline{i}}X^{\underline{i}}\right ]
$$
 \end{Lem}

\begin{remark}
 If $\tilde\eta\in O_0^*$ we can always assume that all $o'_{i0}=0$, because all 
$o'_{i0}X_i$ can disappear by contributing to $f_0$. In other 
words, we assume that $o_{i0}^{\prime 2},o_{i1}^{\prime 2}\in (\tilde\eta O)^{loc}$, 
i.e. these elements belong to $\tilde\eta O$ if $1\leqslant i\leqslant u_0$ 
and belong to the maximal ideal $\m $ of $O$, otherwise. 
\end{remark}

The proof of this Lemma uses the auxiliary statements 
from Subsection \ref{S3.5} and will be given in Subsection \ref{S3.6}  below. 
This is a simplified version of the proof of Main Lemma 
in \cite{Ab}, where we studied the group schemes of period $p>2$. 
As a matter of fact, this simplified version works equally 
well also in the case $p>2$. 
\medskip 

\subsection{Basic strategy II}\label{SS}

Via Main Lemma we shall prove below the existence of $\c M'\in\MF ^{2e}_{S'}$ such that 
$G'\simeq \c G_{O'}(\c M')$. The description of $G_0$ as an element of 
$\Ext _{\Gr _{O_0}}(H_0,\mu _{\eta })$ 
from Subsection \ref{S3.3} is given in terms of 
Kummer's theory and, therefore, is of multiplicative nature. 
On the other hand, the construction of the algebra of $\c G_{O'}(\c M')$ as 
extension of $B'=B\otimes _OO'$ should be given 
(by the definition of  
$\c G_{O'}$) in additive terms. Therefore, the description of $A(G_0)$ 
as an extension of $B_0=A(H_0)$ in terms of the Lubin-Tate group  
is a natural step towards presentation of $G\otimes _OO'$ 
in the form $\c G_{O'}(\c M')$.
\medskip 

 If $G_0$ is given via $f\in\c H_{LT}(H_0,\mu _{\eta })$ then 
$A_0=B_0[Y]$ with equation for $Y$ coming from the relation 
$2l_{LT}(\tilde\eta Y)=l_{LT}(f)$ in $A_0\otimes _{O_0}K_0$. 
If $\mu _{\eta }$ is multiplicative, e.g. $\tilde\eta =1$, 
then it is much easier to obtain an ``integral'' version of the above relation. 
The left-hand side $2l_{LT}(Y)=2Y+Y^2+2(Y^2/2)^2+\dots $ 
looks nicely related to  
operations in filtered modules of the form $\iota _{\c M}(\c M)$. 
As for the right-hand side, we need it to belong 
to $\iota _{\c N}(\c N)$, i.e. to be congruent to a linear combination 
of all $X_i$.  By Main Lemma  
after replacing $f$ by $\tilde f=[f]-[f_0]$, $l_{LT}(\tilde f)$ 
is a linear combination of all $l_{LT}(o'_{i0}X_i)$, 
$l_{LT}(o'_{i1}X_i)$ and $l_{LT}(D_iX_i)$ 
 modulo $2I_{\bar B}$. First two logarithms can give 
non-trivial denominators but in 
$A'=A\otimes _{O}O'$ we can consider the element 
$[g]=[\tilde f]-\sum _{i}([o'_{i0}X_i]+[o'_{i1}X_i])$ and 
because all $D_iX_i\in I_{\bar B}(2)$, $l_{LT}(g)$ is an $O'$-linear 
combination of all $X_1, \dots ,X_u$ modulo $J_{B'}$, i.e. 
gives already an element of 
$\iota _{\c N}(\c N)\otimes _OO'$.  
 
If $\mu _{\eta }$ is 
not multiplicative the calculations should be more precise because 
of the extra factor $\tilde\eta $. (Here we can see a crucial difference 
with the case $p>2$, where ideals of the form $(p/\tilde\eta )I_B$ are  
still $DP$-ideals.)  In 
particular, we can't ignore the quadratic forms in $X_i$ coming from the third term 
$(D_iX_i)^4/4$ of the expansion of 
$l_{LT}(D_iX_i)$ and the second term $g^2$ in $l_{LT}(2g)$, where 
$g=\sum _{\underline{i}}D_{\underline{i}}X^{\underline{i}}$. The 
elaboration of Main Lemma from Subsection \ref{S3.7} 
relates these quadratic forms and allows us to prove in Subsection \ref{S3.8} 
that they, as a matter of fact, 
kill one-another. This provides us in Subsection \ref{S3.10} with explicit 
construction of $\c M'\in\MF ^{2e}_{S'}$ and 
the existence of $\c M\in\MF ^e_S$ such that 
$\c M'=\c M\otimes _SS'$. A formal verification 
that $G\simeq\c G_O(\c M)$ is done in Subsection \ref{S3.11}. 
\medskip

\subsection{Auxiliary statements}\label{S3.5} 

Follow Subsection 3.3 of \cite{Ab} to introduce the ideals $I_B(\alpha )$ and 
$I_B(\alpha )^{loc}$ in $B$, where  $\alpha\in O$. 
Recall that any $a\in I_B$ can be uniquely written as 
$a=\sum _{\underline{i}}o_{\underline{i}}(a)X^{\underline{i}}$ 
with the coefficients $o_{\underline{i}}=o_{\underline{i}}(a)\in O$.

\begin{definition} For any $\alpha\in O$, set 
\medskip 

a)$I_B(\alpha ):=\{a\in I_B\ |\ \text{all} 
\ (o_{\underline{i}}(a)X^{\underline{i}})^2
\in\alpha I_B\}$
\medskip 

b) $I_B(\alpha )^{loc}:=\{a\in I_B\ |\ \text{all}
\ (o_{\underline{i}}(a)X^{\underline{i}})^2
\in\alpha I_B^{loc}\}.$
\end{definition}

Note (use property $({\bf C})$ from Subsection \ref{S1.9.1}), 
that for any $1\leqslant i\leqslant u_0$, 
$X_i\in I_B(\eta _i)^{loc}$ but for $u_0<i\leqslant u$, $X_i\notin 
I_B(\eta _i)^{loc}=I_B^{loc}$. 
In addition,  for arbitrary $\alpha\in O$, 
the $O$-modules $I_B(\alpha )$ and $I_B(\alpha )^{loc}$ 
depend generally on the 
above chosen special construction of $B$. 
Nevertheless, one can verify that:
\medskip  

$\bullet $\ for any $\alpha\in O$, $I_B(\alpha )$ and $I_B(\alpha )^{loc}$ 
are ideals in $B$;
\medskip 

$\bullet $\ for $\alpha |2$,  we have  
$I_B(\alpha )=\{a\in I_B\ |\ a^2\in \alpha I_B\}$ and, similarly,   
$I_B(\alpha )^{loc}=\{a\in I_B\ |\ a^2\in\alpha I_B^{loc}\}$. 
\medskip 

For obvious reasons, the above ideals $I=I_B(\alpha )$ and $I=I_B(\alpha )^{loc}$ 
satisfy property \eqref{AE3} from Subsection \ref{S1.9.3}.

\begin{remark} a) For any $\alpha\in\bar O$ we shall 
denote below by $I_{\bar B}(\alpha )$ and $I_{\bar B}(\alpha )^{loc}$ the 
similar ideals of the $\bar O$-algebra $\bar B=B\otimes _O\bar O$. 
Clearly, they also satisfy property (\ref{AE3}).

b) Using the special basis $\{X^{\underline{i}_1}\otimes 1, 
1\otimes X^{\underline{i}_2}, X^{\underline{i}_1}\otimes X^{\underline{i}_2}\ |\ 
\underline{i}_1, \underline{i}_2\}$ 
of $I_{B\otimes B}$ and $I_{\bar B\otimes\bar B}$ we can introduce 
similarly the ideals $I_{B\otimes B}(\alpha )$ and 
$I_{\bar B\otimes\bar B}(\alpha )$. 
\end{remark}
\medskip

The following lemmas admit straightforward proofs and are quite analogous to 
the lemmas from \cite{Ab} Subsection 6.2, Lemmas 6.2-6.6.  

\begin{Lem} 
\label{L3.4} 
Suppose $C_1,\dots ,C_u\in\bar O$, $g\in I_{\bar B}^{loc}$, 
 $\beta _0\in\bar\m $ and 
$g\equiv\sum _i[C_iX_i]
\operatorname{mod}(I_{\bar B}(\beta _0)^{loc}+\bar{\c I})$, 
where $\bar{\c I}\subset I_{\bar B}^{loc}$ is an ideal 
satisfying condition \eqref{AE3} from Subsection \ref{S1.9.3}. Then 
$$g\equiv \sum _{1\leqslant i\leqslant u}[C_i'X_i]+
\left [\sum _{r(\underline{i})\geqslant 2}
C_{\underline{i}}'X^{\underline{i}}\right ]\operatorname{mod}\bar{\c I}$$
with all $C'_i, C_{\underline{i}}'\in\bar O$, 
$C_iX_i\equiv C_i'X_i\operatorname{mod}I_{\bar B}(\beta _0)^{loc}$ and 
$C_{\underline{i}}X^{\underline{i}}\in I_{\bar B}(\beta _0)^{loc}$.
\end{Lem}

\begin{Lem} 
\label{L3.5} 
Suppose $l$ equals either $0$ or $1$.  

{\rm a)} If $o_1',o_2'\in\pi ^{\prime l}O$ are such that  
$o_1^{\prime 2}, o_2^{\prime 2}\in\tilde\eta O$,  
then for any $a\in I_B$, 
$[o_1'a]+[o_2'a]-[(o_1'+o_2')a]\in\tilde\eta I_B$;

{\rm b)} If  $o'\in\pi ^{\prime l}O$, 
$o^{\prime 2}\in\tilde\eta O$ and $a_1,a_2\in I_B$ then 
\newline 
$[o'a_1]+[o'a_2]-[o'(a_1+a_2)]\in\tilde\eta I_B$. 
\end{Lem}

\begin{remark} a) When proving Lemma \ref{L3.4} use first remark from 
Subsection \ref{S1.9.3}; 
b) note the following special cases of above Lemma \ref{L3.5}:

---\ $[o'a]+[-o'a]\in\tilde\eta I_B$;

---\ $[\delta _{LT}(o'X_i)]-[o'j_i]\in\tilde\eta I_{B\otimes B}$. 
\end{remark}

\begin{Lem} 
\label{L3.6} 
If $C\in\bar O$, $\alpha _1\in\bar\m $ 
 and $CX_i\in I_{\bar B}(\alpha _1)^{loc}$ then 
$$\delta _{LT}(CX_i)\equiv C^2X_i\otimes X_i
\operatorname{mod}I_{\bar B}(\alpha _1^4)^{loc}+\bar{\c J},$$
where $\bar{\c J}$ is the ideal defined in 
the end of Subsection \ref{S1.9.3}. 
\end{Lem} 

\begin{remark}
 Note that if $CX_i\in I_{\bar B}(2)^{loc}$ then $C^2X_i\otimes X_i\in\bar{\c J}$. 
\end{remark}

\begin{proof} From the definition of 
the Lubin-Tate group law, cf. Subsection \ref{S2.2.3}  it follows that 

 \begin{equation}\label{E3.1}
\delta _{LT}(CX_i)=
[Cj_i]-[-C^2(X_i\otimes 1+1\otimes X_i)j_i]-[-C^2X_i\otimes X_i]
  \end{equation} 
$$-\sum _{n\geqslant 2}[C^{2^n}P_n(X_i\otimes 1+1\otimes X_i, j_i)]
-\sum_{n\geqslant 2}[C^{2^n}P_n(X_i\otimes 1,1\otimes X_i)]$$

We know  that for $1\leqslant i\leqslant u_0$, $j_i\in \bar{\c J}$
 and if $u_0<i\leqslant u$ then 
$C\in\bar\m $ and again $Cj_i\in \bar{\c J}$. Therefore, 
the first, second and forth terms of the right-hand side 
of (\ref{E3.1}) belong to $\bar{\c J}$. 
As for the last term of that formula 
it remains to note that 
$\left (I_{\bar B}(\alpha _1)^{loc}\right )^4\subset I_{\bar B}(\alpha _1^4)^{loc}$. 
\end{proof} 

 \subsection{Proof of Main Lemma} \label{S3.6} 

Prove that for any $\alpha\in\bar\m $, one has  

\begin{equation} \label{E3.2} 
 f\equiv [f_{\alpha }]+
 \sum _{1\leqslant i\leqslant u}\left ([o_{i0}'X_i]+
[o'_{i1}X_i]+[D_iX_i]\right )\operatorname{mod}
(2I_{\bar B}^{loc}+I_{\bar B}(\alpha ^2)^{loc}) 
\end{equation} 
where  

$\bullet $\ $f_\alpha\in \tilde\eta I_B$;

\ \ 

$\bullet $\ all $o'_{i0}=o'_{i0}(\alpha )\in O$, $o'_{i1}=
o'_{i1}(\alpha )\in\pi 'O$ are such that 
$o^{\prime 2}_{i0}, o^{\prime 2}_{i1}\in\tilde\eta O$; 
(If $\tilde\eta\in O_0^*$ then we can assume that all $o'_{i0}=0$.) 

\ \ 

$\bullet $\ all $D_i\in\bar O$ are such that 
$D_iX_i\in I_{\bar B}(\alpha ,2)^{loc}:=I_{\bar B}(\alpha )^{loc}
+I_{\bar B}(2)^{loc}$.
\medskip 

First, 
 there is an $\alpha _0\in\m $  
such that \eqref{E3.2}  holds for trivial reasons with $\alpha =\alpha _0$. 
Indeed, if $\tilde\eta\notin O_0^*$ then 
use that $f^2\equiv [2](f)\equiv 0\operatorname{mod}\tilde\eta I_{\bar B}$; 
otherwise, use that $[f]-[f_0]\in\bar\m I_{\bar B}$, 
where $f_0\in I_{B}$ and $f\equiv f_0\operatorname{mod}\bar\m I_{\bar B}$. 

By induction on $\alpha $ it will be sufficient to prove that if 
(\ref{E3.2}) holds with $\alpha =\alpha _1\in\bar\m $ then it also holds with 
$\alpha =\alpha _1^2$. 

Apply Lemma \ref{L3.4} with $\bar{\c I}=2I_{\bar B}^{loc}$ to obtain 
\begin{equation}\label{AAE1}
f=[f_{\alpha _1}]+\sum _{1\leqslant i\leqslant u}\left ([o_{i0}'X_i]+
[o'_{i1}X_i]+[D'_iX_i]\right )+\left [\sum _{r(\underline{i})\geqslant 2}
D_{\underline{i}}'X^{\underline{i}}\right ]\operatorname{mod}2I_{\bar B}^{loc},
\end{equation} 
where all $D_i', D_{\underline{i}}'\in\bar O$, 
$D_i'X_i\in I_{\bar B}(\alpha _1,2)^{loc}$ and 
$D_{\underline{i}}'X^{\underline{i}}\in I_{\bar B}(\alpha _1^2)^{loc}$.  

By Lemma \ref{L3.5} all $\delta _{LT}(o'_{i0}X_i), 
\delta _{LT}(o'_{i1}X_i)\in\tilde\eta I_{B\otimes B}+\bar{\c J}$. Therefore, 
the condition $\delta _{LT}(f)\in\tilde\eta I_{B_0\otimes B_0}
\subset \tilde\eta I_{B\otimes B}$ implies  
(use Lemma \ref{L3.6}) that 

$$\sum _{1\leqslant i\leqslant u}D_i^{\prime 2}X_i\otimes X_i+
\sum _{\underline{i}_1+\underline{i}_2=\underline{i}}
D_{\underline{i}}'X^{\underline{i}_1}\otimes X^{\underline{i}_2}
\in\tilde\eta I_{B\otimes B}
\operatorname{mod}(I_{\bar B\otimes\bar B}(\alpha _1^4)^{loc}+\bar{\c J})
$$
(Recall that all multi-indices $\underline{i}$, $\underline{i}_1$, 
$\underline{i}_2$ are non-zero vectors with coordinates $0$ or $1$. )

This implies the following properties 
(cf. first Remark in Subsection \ref{S1.9.3}):

a) for $1\leqslant i\leqslant u$, 
 $D_i^{\prime 2}X_i\otimes X_i
\equiv o_iX_i\otimes X_i\operatorname{mod}(I_{\bar B\otimes \bar B}
(\alpha _2^2)^{loc}+\bar{\c J})$, where $o_i\in\tilde\eta O$; 
\medskip 

b) if $r(\underline{i})\geqslant 2$ then 
 $D'_{\underline{i}}X^{\underline{i}_1}\otimes X^{\underline{i}_2}
\equiv o_{\underline{i}}X^{\underline{i}_1}\otimes X^{\underline{i}_2}
\operatorname{mod}(I_{\bar B\otimes\bar B}(\alpha _2^2)^{loc}+\bar{\c J})$, where  
$o_{\underline{i}}\in\tilde\eta O$. 
\medskip  

Consider the morphism of multiplication $m:\bar B\otimes\bar B\To\bar B$. 
Then $m(I_{\bar B\otimes\bar B}(\alpha _2^2)^{loc})
=I_{\bar B}(\alpha _2^2)^{loc}$, $m(\bar{\c J})=2I_{\bar B}^{loc}$. 
Therefore, 
b) implies that 
$(D_{\underline{i}}'-o_{\underline{i}})X^{\underline{i}}
\in I_{\bar B\otimes\bar B}(\alpha _2^2)^{loc}+2I_{\bar B}^{loc}$ and 
the last summand in (\ref{AAE1}) disappears 
modulo $I_{\bar B}(\alpha _2^2)^{loc}$ by contributing to 
the corresponding $f_{\alpha _2}\in\tilde\eta I_B$. If $2|\alpha _1$ formula 
$\eqref{E3.2}$ has been already proved for $\alpha =\alpha _2$ because we 
don't need to change  
the terms $D_i'X_i$. 

If $\alpha _1|2$ we should continue with property a). 
If $D_i'X_i\otimes X_i\in\bar{\c J}$ keep $D_i'$ the same. 
 Otherwise, by first Remark from  Subsection \ref{S1.9.3} we can assume that 
$D_i'X_i\otimes X_i\in I_{\bar B\otimes\bar B}(\alpha _2^2)^{loc}$. 
Consider the elements 
$o''_{i0}\in O$ and $o''_{i1}\in\pi 'O$ 
such that $o_i\equiv o_{i0}''^2+o_{i1}''^2
\operatorname{mod}2o_i\pi $ and 
$o_{i0}''^2, o_{i1}''^2\equiv 0\operatorname{mod}o_i$.

\begin{Lem} \label{L3.7} 
 $[D_i'X_i]+[-o''_{i0}X_i]+[-o''_{i1}X_i]\equiv 
\sum _j[D_{ij}'X_j]\operatorname{mod}I_{\bar B}(\alpha _2^2)^{loc},$
where all $D_{ij}'\in\bar O$ and $D_{ij}'X_j\in I_{\bar B}(\alpha _2^2)^{loc}$. 
\end{Lem}

\begin{proof}
 Suppose $1\leqslant i\leqslant u_0$. In this case $X_i^2\in I_{B}(\eta _i)^{loc}$ and 
property a) means that 
$(D_i'^2-o_i)\eta _i\in\alpha _2\bar O$. Therefore (use that $\alpha _2|4$), 
$(D_i'-o_{i0}''-o_{i1}'')^2\eta _i\in\alpha _2\bar O$, i.e. 
$(D_i'-o_{i0}''-o_{i1}'')X_i\in I_{\bar B}(\alpha _2)^{loc}$. Also, 
$0\equiv D_i'^2\eta _i\equiv o_i\eta _i\equiv o_{i0}''^2\eta _i
\equiv o_{i1}''^2\eta _i\operatorname{mod}\alpha _1\bar O$,  
i.e. $o''_{i0}X_i, o_{i1}''X_i\in I_{\bar B}(\alpha _1)^{loc}$.

Suppose $u_0<i\leqslant u$. In this case $\eta _i\in O^*$, $D_i\in\bar\m $ and 
$X_i, X_i^2\in I_{B}^{et}$. The relation $(D_i'^2-o_i)X_i
\otimes X_i\in I_{\bar B}(\alpha _2^2)^{loc}$ means that 
$(D_i'^2-o_i)\eta _i\in\alpha _2\bar\m $, 
$D_i'^2\eta _i\equiv (o_{i0}''^2+o_{i1}''^2)\eta _i\operatorname{mod}\alpha _2\bar\m $, 
$(D_i'-o_{i0}''-o_{i1}'')^2\eta _i\in\alpha _2\bar\m $ and again 
$(D_i-o_{i0}''-o_{i1}'')X_i\in I_{\bar B}(\alpha _2^2)^{loc}$. Similarly, 
$D_i'X_i\in I_{\bar B}(\alpha _1)^{loc}$ means that 
$0\equiv D_i'^2\eta _i\equiv o_i\eta _i\equiv o_{i0}''^2\eta _i\equiv o_{i1}''^2\eta _i
\operatorname{mod}\alpha _1\bar\m $, i.e. 
$o_{i0}''X_i, o_{i1}''X_i\in I_{\bar B}(\alpha _1)^{loc}$.

Now for any $1\leqslant i\leqslant u$, 
$$[D_i'X_i]+[-o''_{i0}X_i]+[-o''_{i1}X_i]\equiv [(D_i'-o''_{i0}-o''_{i1})X_i]$$
$$-[D_i'o''_{i0}X_i^2]-[D_i'o''_{i1}X_i^2]-[-o''_{i0}o''_{i1}X_i^2]\operatorname{mod}
I_{\bar B}(\alpha _2^2)^{loc}$$
where all terms in the right-hand side belong to 
$\left (I_{\bar B}(\alpha _1)^{loc}\right )^2\subset I_{\bar B}(\alpha _2)^{loc}$.  
Thus  the right-hand side is congruent modulo $I_{\bar B}(\alpha _2^2)^{loc}$ to 
$\sum _j[D'_{ij}X_j]$ with all $D'_{ij}\in\bar O$ and 
$D'_{ij}X_j\in I_{\bar B}(\alpha _2)^{loc}$. Lemma \ref{L3.7} is proved.
\end{proof}
Finally, we finish the proof of formula (\ref{E3.2}) with $\alpha =\alpha _2$ 
by noting that (use Lemma \ref{L3.5}) 
$$\sum _i([o'_{i0}X_i]+[o'_{i1}X_i])-
\sum _i([-o''_{i0}X_i]+[-o''_{i1}X_i])=$$
$$[\tilde a]+\sum _i([(o'_{i0}+o''_{i0})X_i]+[(o'_{i1}+o''_{i1})X_i]),$$
where $\tilde a\in\tilde\eta I_B$. 

Clearly, there is $\alpha\in\bar\m $ such that $I_{\bar B}(\alpha ^2)^{loc}
\subset 2I_{\bar B}^{loc}$. 
It remains to apply   
Lemma \ref{L3.4} with $\bar{\c I}=0$ to 
finish the proof of Main Lemma.
\medskip

\subsection{Elaboration of Main Lemma} 
\label{S3.7} 

Let $D=C^{-1}=(d_{ij})$ and for $1\leqslant i\leqslant u$, 
$$R_i=\sum _{\substack{t \\ s_1<s_2 }}(X_{s_1}\otimes X_{s_2}
+X_{s_2}\otimes X_{s_1})
c_{s_1t}c_{s_2t}d_{ti}^2\in I_{B\otimes B}.$$
Then  property $\bf (C)$ of 
the matrix $C$ from Subsection \ref{S1.9.1} implies that all 
$\tilde\eta _iR_i\in I_{B\otimes B}^{loc}$. Indeed, if 
$\tilde\eta _iX_{s_1}\otimes X_{s_2}\notin I_{B\otimes B}^{loc}$ then 
$1\leqslant i\leqslant u_0$ (because $\tilde\eta _i$ must belong to $O_0^*$) 
and $u_0<s_1,s_2\leqslant u$ 
(because $X_{s_1}, X_{s_2}\notin I_{B}^{loc}$), but then $c_{s_1t}d_{ti}, 
c_{s_2t}d_{ti}\equiv 0
\operatorname{mod}\pi $. Therefore, the system of congruences 
\begin{equation} 
\label{E3.5} 
 \sum _sB_sc_{si}\equiv \tilde\eta _i(R_i+B_i^2)
\operatorname{mod}I_{B\otimes B}(4)^{loc}, \ \ 1\leqslant i\leqslant u,  
\end{equation}
has a unique solution 
$(B_1,\dots ,B_u)\operatorname{mod}I_{B\otimes B}(4)^{loc}$ 
with all $B_i\in I_{B\otimes B}^{loc}$.  

We shall use below the following agreement: if $1\leqslant i\leqslant u$ 
and $\alpha\in O$ then $I_B(\alpha )^{(i)}$ will be equal to 
$I_B(\alpha )^{loc}$ if $i\leqslant u_0$ and to $I_B(\alpha )$ if $i>u_0$. 
Same agreement will be used for similar ideals 
$I_{B\otimes B}(\alpha )$, $I_{\bar B}(\alpha )$, etc. 

\begin{Lem} \label{L3.8} 
 With above notation one has 
 $$
j_i\equiv \sum _t\tilde\eta _tX_t\otimes X_td_{ti}+2B_i
+\sum _{t,u}\tilde\eta _u(X_u\otimes X_u)d_{ut}
(\tilde\eta _tX_t\otimes 1+1\otimes\tilde\eta _tX_t)d_{ti}
$$
$$\qquad\qquad +2\sum _t(\tilde\eta _t
X_t\otimes 1+1\otimes \tilde\eta _tX_t)B_td_{ti}
\operatorname{mod}I_{B\otimes B}(16)^{(i)}$$
\end{Lem}

\begin{proof} It will be  
sufficient to verify by direct calculations 
that the right-hand sides of above congruences give solutions of  
equalities (\ref{E1.1}) modulo $I_{B\otimes B}(16)^{(i)}$. 
(When calculating use that the first two summands belong to 
$I_{B\otimes B}(4)^{(t)}$ and  the last two ones -- to 
$I_{B\otimes B}(8)^{(t)}$.) 
\end{proof} 

One can easily see that the above elements 
$B_i\in I_{B\otimes B}^{loc}$ appear in the form 
$$B_i\equiv \sum _{s_1<s_2}
\gamma _{s_1s_2i}(X_{s_1}\otimes X_{s_2}+
X_{s_2}\otimes X_{s_1})\operatorname{mod}I_{B\otimes B}(4)^{loc}$$
where all $\gamma _{s_1s_2i}\in O$.  
Introduce  the elements 
$$\widetilde{R}_i=\sum _{\substack{t \\ s_1<s_2 }}
X_{s_1}X_{s_2}c_{s_1t}c_{s_2t}d^2_{ti}, \ \ 
\widetilde{B}_i=\sum _{s_1<s_2 }\gamma _{s_1s_2i}
X_{s_1}X_{s_2}$$

Then 
\medskip  

$\bullet $\ $(\widetilde{B}_1,\dots ,\widetilde{B}_u)\operatorname{mod}
I_{B}(4)^{loc}$ is a unique solution in $I_B^{loc}\operatorname{mod}I_B(4)^{loc}$ 
of the congruences 
$\sum _s\widetilde{B}_sc_{si}\equiv\tilde\eta _i
(\widetilde{R}_i+\widetilde{B}_i^2)\operatorname{mod}I_{B}(4)^{loc}$ 
;
\medskip

$\bullet $\ for all $i$, $\delta ^+(\widetilde{B}_i)\equiv 
B_i\operatorname{mod}I_{B\otimes B}(4)^{loc}$. 
\medskip 

Now we can state the following elaboration of Main Lemma. 

\begin{Prop} 
\label{P3.9} 
 In Lemma {\rm \ref{L3.3}} the elements $D_i\in\bar O$, 
$1\leqslant i\leqslant u$, and 
$g:=\sum _{r(\underline{i})\geqslant 2}
D_{\underline{i}}X^{\underline{i}}
\in I_{\bar B\otimes\bar B}$ can be taken in such a way that   

{\rm a)} $(D_i^2+\sum _s(o'_{s0}+o'_{s1}+D_s)
d_{is}\tilde\eta _i- \tilde o_i)X_i^2\in 4I_{\bar B}$, 
where $1\leqslant i\leqslant u$ 
and all $\tilde o_i\in\tilde\eta O$;

{\rm b)}\ $g\equiv\sum _i(o'_{i0}+o'_{i1}+D_i)\widetilde{B}_i
\operatorname{mod}I_{\bar B}(4)^{loc}. 
$
\end{Prop}

\begin{remark} Part a) implies that 
all $\tilde o_i\equiv 0\operatorname{mod}\tilde\eta _i$ and  
if $u_0<t\leqslant u$ then 
$\tilde o_i\equiv 0\operatorname{mod}\pi $. In other words, all 
$\tilde o_i\in (\tilde\eta _iO)^{loc}$. 
\end{remark}
\medskip

\begin{proof} 
Let $\tilde f :=[f]-[f_0]$. Then 
$$\delta _{LT}(\tilde f)\equiv \sum _i(o'_{i0}+o'_{i1}+D_i)j_i+
\sum _iD_i^2X_i\otimes X_i\qquad\qquad $$
$$\qquad +\sum _iD_i^2(X_i\otimes 1+1\otimes X_i)j_i+2\delta ^+(g)
\operatorname{mod}I_{\bar B\otimes\bar B}(16)^{loc}.$$

For $1\leqslant t\leqslant u$, let 
$s(t)=\sum _i(o'_{i0}+o'_{i1}+D_i)d_{ti}\tilde\eta _t$. 
Then using the explicit formulas for  
$j_i\operatorname{mod}\,I_{B\otimes B}(16)^{(i)}$, 
$1\leqslant i\leqslant u$, 
 from Lemma \ref{L3.8} we obtain that 
$\delta _{LT}(\tilde f)$ is congruent 
modulo $I_{\bar B\otimes\bar B}(16)^{loc}$ to 
$$\sum _t(D_t^2+s(t))X_t\otimes X_t+
2\sum _t(o'_{t0}+o'_{t1}+D_t)B_t+\qquad\qquad\qquad\qquad\qquad $$ 
$$\sum _t(D_t^2+s(t))(X_t\otimes 1+1\otimes X_t)
\left (2B_t+\sum _ud_{ut}\tilde\eta _uX_u\otimes X_u\right )+2\delta ^+(g). $$

Follow the coefficient for $X_i\otimes X_i$, $1\leqslant i\leqslant u$. 

Verify that only the first sum contributes to this coefficient. 

Indeed, the second sum does not contribute because  
$B_t\operatorname{mod}\,I_{B\otimes B}(4)^{loc}$ is a linear combination of 
the terms $X_{s_1}\otimes X_{s_2}$ with $s_1\ne s_2$. 
The remaining big sum also does not contribute modulo $4I_{\bar B\otimes\bar B}^{loc}$ 
because:

--- $2B_t(X_t\otimes 1)$ and $2B_t(1\otimes X_t)$ contribute in the same way;

--- for similar reason it will be sufficient to verify that 
$X_i\otimes X_i$ can appear in 
$(\tilde\eta _uX_u\otimes X_u)(X_t\otimes 1)$ with coefficient from $2O$; 
but if this coefficient is not 0 then 
$u=t$ and we can use that $\eta _u\tilde\eta _u=-2$. 

So, for $1\leqslant i\leqslant u$,  
$(D_i^2+s(i))X_i\otimes X_i\in\tilde\eta I_{B\otimes B}\operatorname{mod}
I_{\bar B\otimes\bar B}(16)^{loc}$. 
Therefore, there is $\tilde o_i\in\tilde\eta O$ such that 
$(D_i^2+s(i)-\tilde o_i)\eta _i$ belongs to 
$4\bar O$.  
This proves part a). 
 
The remaining terms in the relation $\delta _{LT}(\tilde f)
\in\tilde\eta I_{B\otimes B}
\operatorname{mod}I_{\bar B\otimes\bar B}(16)^{loc}$ 
give 
$$\sum _i(o'_{i0}+o'_{i1}+D_i)B_i+\delta ^+(g)\in I_{B\otimes B}
\operatorname{mod}I_{\bar B\otimes\bar B}(4)^{loc}$$
(use property a) to eliminate the last big sum). 

We know that $h=\sum _i(o'_{i0}+o'_{i1}+D_i)\widetilde{B}_i+g$ 
is of the form $\sum _{r(\underline{i})\geqslant 2}C_{\underline{i}}
X^{\underline{i}}$. Therefore, the relation 
$\delta ^+h\in I_{B\otimes B}
\operatorname{mod}I_{\bar B\otimes\bar B}(4)^{loc}$ implies that 
$h\equiv h_0\operatorname{mod}I_{\bar B}(4)^{loc}$ with $h_0\in I_{B}$. 
So, replacing $f_0$ by $f_0-2h_0$ we obtain property b).
\end{proof}
\medskip 

\subsection{Explicit calculation of $l_{LT}(f)$} 
\label{S3.8} 

As earlier, $f\in\c H_{LT}(H_0,\mu _{\eta })$ is 
given via  Main Lemma where we can now assume that $D_1,\dots ,D_u\in \bar O$ 
and $g\in I^{loc}_{\bar B}$ satisfy Proposition \ref{P3.9}.  
We also set $\tilde f=[f]-[f_0]$,   
$L_{\bar B}=\sum _i\bar OX_i$,  
$L^{loc}_{\bar B}=\sum _{1\leqslant i\leqslant u_0}\bar OX_i+
\sum _{u_0<i\leqslant u}\bar\m X_i$ and define   
$$(\tilde\eta L_{\bar B})^{loc}=
\begin{cases} \tilde\eta L_{\bar B} & \mbox{if}\ \tilde\eta\notin O^* \\
 L_{\bar B}^{loc} & \mbox{if}\ \tilde\eta\in O^* 
\end{cases}
\qquad 
(\tilde\eta I_{\bar B}(4))^{loc}=
\begin{cases} \tilde\eta I_{\bar B}(4) & \mbox{if}\ \tilde\eta\notin O^* \\
 I_{\bar B}(4)^{loc} & \mbox{if}\ \tilde\eta\in O^* 
\end{cases}
$$
For $1\leqslant i\leqslant u$ and  
$\tilde o_i\in O$ from Proposition {\rm \ref{P3.9}},  
let $o'_{i2}\in O$ and $o'_{i3}\in\pi 'O$ be such that 
$\tilde o_i\equiv  o_{i2}'^2+o_{i3}'^2
\operatorname{mod}2\tilde\eta \pi O$ and 
$o_{i2}'^2, o_{i3}'^2\in\tilde o_iO$.

\begin{Prop} 
\label{P3.10} With above notation 
$$l_{LT}(\tilde f)\equiv  -\sum _{i,l}o'_{il}X_i+
\sum _{i,l}l_{LT}(o'_{il}X_i)
\operatorname{mod}((\tilde\eta L_{\bar B})^{loc}
+(\tilde\eta I_{\bar B}(4))^{loc}).$$ 
(In both sums $1\leqslant i\leqslant u$ and $0\leqslant l\leqslant 3$.)
\end{Prop}

\begin{proof} Note that (use that all $D_iX_i\in I_{\bar B}(2)$)
$$l_{LT}(D_iX_i)\equiv D_iX_i+\frac{D_i^2X_i^2}{2}+
\left (\frac{D_i^2X_i^2}{2}\right )^2+
2\left (\frac{D_i^2X_i^2}{2}\right )^4\operatorname{mod}4I_{\bar B}$$
Rewrite property a) of Proposition \ref{P3.9} in the following form 
$$\frac{D^2_iX_i^2}{2}\equiv\frac{\tilde o_iX_i^2}{2}
-\sum_{s}(o'_{s0}+o'_{s1}+D_s)d_{is}\tilde\eta _iX_i^2
\operatorname{mod}2L_{\bar B}$$

Then 
$$\sum _i\frac{D_i^2X_i^2}{2}\equiv 
\sum _i\frac{\tilde o_iX_i^2}{2}-\sum _s(o'_{s0}+o'_{s1}+D_s)
X_s\operatorname{mod}2L_{\bar B}$$
Property a) also implies that  
$$\left (\frac{D_i^2X_i^2}{2}\right )^2\equiv 
\left (\frac{\tilde o_i^2}{\tilde\eta _i}-\sum_s(o'_{s0}+o'_{s1}+D_s)
d_{is}\right )^2\left (X_i^2/\eta _i\right )^2
\operatorname{mod}4I_{\bar B}$$

The first factor in the right-hand side of the last congruence 
can be written in the form 
$$\left (\frac{o^{\prime 2}_{i2}}{\tilde\eta _i}\right )^2+
\left (\frac{o^{\prime 2}_{i3}}{\tilde\eta _i}\right )^2+
\sum _sD_s^2d^2_{is}+\tilde\eta A_i$$
where $A_i\in\bar O$ and 
belongs to the maximal ideal $\bar\m $ of $\bar O$ 
if $u_0<i\leqslant u$. Then 

$$\left (\frac{D_i^2X_i^2}{2}\right )^2\equiv 
\left (\frac{o_{i2}'^2X_i^2}{2}\right )^2+
 \left (\frac{o_{i3}'^2X_i^2}{2}\right )^2
+2\sum _{\substack{s\\ u_1<u_2}} D_s^2d_{is}^2X_{u_1}X_{u_2}c_{u_1i}c_{u_2i}$$
modulo $(\tilde\eta L_{\bar B})^{loc}
+2(\tilde\eta I_{\bar B})^{loc}$, because  
$$\sum _{i,s,u}D_s^2d^2_{is}X_u^2c^2_{ui}\equiv 
\sum _sD_s^2X_s^2\equiv 0\operatorname{mod}2L_{\bar B}.$$

Note also that 
$$2\left (\frac{D_i^2X_i^2}{2}\right )^4
=2\left (\frac{D_i^2\eta _i}{2}\right )^4\left (\sum _sX_s^2c^2_{si}
+2\sum _{s_1<s_2}X_{s_1}X_{s_2}c_{s_1i}c_{s_2i}\right )^2
$$
belongs to $2L_{\bar B}+4I_{\bar B}$. 

Taking above calculations together we obtain  
$$\sum _il_{LT}(D_iX_i)\equiv 
-\sum _i(o'_{i0}+
o'_{i1})X_i+
\sum _{i,l}\frac{o_{il}'^2X_i^2}{2}
+2\sum _sD_s^2\widetilde{R}_s\equiv $$
$$
-\sum_{i,l}o'_{il}X_i+
\sum _il_{LT}([o_{i2}'X_i]+[o_{i3}'X_i])
+2\sum _sD_s^2\widetilde{R}_s$$ 
modulo $(\tilde\eta L_{\bar B})^{loc}
+2(\tilde\eta I_{\bar B})^{loc}$ 
because all 
$2(o_{i2}^{\prime 2}X_i^2/2)^4$ and  
$2(o_{i3}^{\prime 2}X_i^2/2)^4$ belong to 
$2L_{\bar B}+4I_{\bar B}$. 

On the other hand we have:  
\medskip 

$g^2\equiv \sum _tD_t^2\wt{B}_t^2\operatorname{mod}(\tilde\eta I_{\bar B})^{loc}$, 
use Proposition \ref{P3.9}b);
\medskip 

$\sum _tD_t^2\wt{B}_t^2\equiv \sum _{t,s}(D_t^2/\tilde\eta _t)\wt{B}_sc_{st}
-\sum _tD_t^2\wt{R}_t\operatorname{mod}I_{\bar B}(4)^{loc}$, 
cf. Subsection \ref{S3.8}; 
\medskip 

$D_t^2/\tilde\eta _t\equiv -\sum _s(o'_{s0}+o'_{s1}+D_s)
d_{is}\operatorname{mod}\tilde\eta \bar O$, 
use Proposition \ref{P3.9}a).

Therefore, 
$$l_{LT}(2g)\equiv 2(g^2+g)\equiv -2\sum _tD_t^2\wt{R}_t
\operatorname{mod}(\tilde\eta I_{\bar B}(4))^{loc}$$
because $2(\tilde\eta I_{\bar B})^{loc}$, 
$2I_{\bar B}(4)^{loc}$, $2\tilde\eta I_{B\otimes B}^{loc}$ are contained in 
$(\tilde\eta I_{\bar B}(4))^{loc}$. 

The Proposition is proved. 
\end{proof} 
\medskip

\subsection {Special element $h\in I_{\bar B}$ and its properties} 
\label{S3.9} 

With above notation set 

$$h=[\tilde f]-\sum _{i,l}[o'_{il}X_i].$$

Note that $[2](h)=[2](f)-[2](f_0)-\sum _{i,l}[2](o'_{il}X_i)\in I_{B'}
\subset I_{\bar B}$ 
and, therefore, $l_{LT}(h)=(1/2)l_{LT}([2]h)\in I_{B'}\otimes _OK$. Then by 
Proposition \ref{P3.10}  
$$l_{LT}(h)\equiv\tilde\eta 'l_0\operatorname{mod}(\tilde\eta I_{B'}(4))^{loc},$$
where $l_0=\sum _ia_iX_i$ with  
$a_i\in O'$ and one has  
\begin{equation}\label{AAE3}
\tilde\eta 'a_i\equiv -\sum _lo'_{il}
\operatorname{mod}\tilde\eta O'.
\end{equation}

In addition, if $\tilde\eta\in O_0^*$ then $l_0\in I_{B'}^{loc}$. 

\begin{Prop} 
\label{P3.11} 

{\rm a)} $[2](h)\equiv 2\tilde\eta 'l_0
\operatorname{mod}2(\tilde\eta I_{B'})^{loc}$;
\medskip 

{\rm b)}\ $\delta _{LT}(h)\equiv \tilde\eta ' \delta ^+(l_0)
\operatorname{mod}(\tilde\eta 'I_{B'\otimes B'}(4))^{loc}$.
\end{Prop} 

\begin{proof} a) By Main Lemma, $h\in I_{B'}(2)^{loc}$ and, therefore,  
$[2](h)\in 2I_{B'}(2)^{loc}+
\left (I_{B'}(2)^{loc}\right )^2\subset I_{B'}(4)^{loc}$, cf. 
Subsection \ref{S2.2.3}. The standard DP-structure on this ideal is topologically 
nilpotent and $2(\tilde\eta 'I_{B'})^{loc}$  
is  its DP-subideal. Therefore, the congruence 
 $l_{LT}([2]h)\equiv 2\tilde\eta '
l_0\operatorname{mod}2(\tilde\eta I_{B'}(4))^{loc}$ implies that 
$[2](h)\in 2(\tilde\eta 'I_{B'})^{loc}$ (use that $l_{LT}$ is bijective on any 
ideal with nilpotent divided powers). If $[2](h)=2A$ with 
$A\in (\tilde\eta 'I_{B'})^{loc}$ then 
$l_{LT}([2](h))\equiv 2A\operatorname{mod}2(\tilde\eta I_{B'})^{loc}$ and, 
therefore,  
$A\equiv \tilde\eta 'l_0\operatorname{mod}(\tilde\eta I_{B'})^{loc}$. 

b) Similarly, Main Lemma implies that 
$$\delta _{LT}(h)\in I_{B'}(2)^{loc}\otimes I_{B'}(2)^{loc}
\subset I_{B'\otimes B'}(4)^{loc}.$$
Again this ideal has topologically nilpotent DP-structure 
and in addition we have  
$$l_{LT}(\delta _{LT}(h))=\delta ^+l_{LT}(h)\equiv \tilde\eta '\delta ^+l_0
\operatorname{mod}(\tilde\eta I_{B'\otimes B'}(4))^{loc}$$
Proceeding as in a) we obtain $\delta _{LT}(h)\in (\tilde\eta 'I_{B'\otimes B'}(4))^{loc}$.  
It remains to note that then  
$\delta _{LT}(h)^2/2\in (\tilde\eta I_{B'\otimes B'}(4))^{loc}$ and, therefore,  
$l_{LT}(\delta _{LT}(h))\equiv\delta _{LT}(h)
\operatorname{mod}(\tilde\eta I_{B'\otimes B'}(4))^{loc}$. 
\end{proof}

\begin{Prop} 
\label{P3.12} 
 For $1\leqslant i\leqslant u$, 
$\left (\sum _{l}o_{il}'^2\right )\eta _i\in\tilde\eta ^2O_0
\operatorname{mod}2\tilde\eta O.$
\end{Prop}

\begin{proof} 
We know that $[2](f)\in\tilde\eta ^2I_{B_0}$. 
On the other hand, 
$$[2](f)=[2](f_0)+[2](h)+\sum _{i,l}[2](o'_{il}X_i),$$
where $f_0\in\tilde\eta I_B$. Note that:
\medskip 

 $[2](f_0)\equiv f_0^2\operatorname{mod}2\tilde\eta I_B$ implies that 
$[2](f_0)\in\tilde\eta ^2I_{B_0}\operatorname{mod}2(\tilde\eta I_{B})^{loc}$; 
\medskip 

 $[2](o'_{il}X_i)\equiv [o_{il}'^2X_i^2]+[2o'_{il}X_i]\operatorname
{mod} 2(\tilde\eta I_B)^{loc}$ 
(use that all $o_{il}'^2\in\tilde\eta O$);
\medskip 

 $[2](h)\equiv -\sum _{i,l}[2o'_{il}X_i]
\operatorname{mod}2(\tilde\eta I_{B'})^{loc}$ 
by the above proposition.  
\medskip 

So, 
\begin{equation}\label{AAE2} 
\sum _{i,l}[o_{il}'^2X_i^2]\in 
\tilde\eta ^2I_{B_0}\operatorname{mod}2(\tilde\eta I_B)^{loc}. 
\end{equation} 
Note that for all $i$ (use just that $\pi ^2=\pi _0$), 
\medskip 

$o_{i0}'^2, o_{i2}'^2\in\tilde\eta O_0+2\tilde\eta O$; 
\medskip 
  
 $o_{i1}'^2, o_{i3}'^2\in\pi\tilde\eta O_0+2\tilde\eta O$; 
\medskip 
 
$X_i^2\in\eta _iI_{B_0}+2I_B$.
\medskip 

This implies that:
$$\sum _i\left ([o_{i0}'^2X_i^2]+[o_{i2}'^2X_i^2]\right )
\equiv\sum _i\left (o_{i0}'^2+o_{i2}'^2\right )X_i^2\operatorname{mod}
(\tilde\eta ^2I_{B_0}+2\tilde\eta I_B)$$
$$\sum _i\left ([o_{i1}'^2X_i^2]+[o_{i3}'^2X_i^2]\right )
\equiv\sum _i\left (o_{i1}'^2+o_{i3}'^2\right )X_i^2\operatorname{mod}
(\tilde\eta ^2I_{B_0}+2\tilde\eta I_B)$$

Clearly, 
$\sum _i\left (o_{i0}'^2+o_{i2}'^2\right )X_i^2\in\tilde\eta I_{B_0}+2\tilde\eta I_B$ 
and 
$\sum _i\left (o_{i0}'^2+o_{i2}'^2\right )X_i^2\in\pi \tilde\eta I_{B_0}+2\tilde\eta I_B$. 
Therefore, (\ref{AAE2}) implies that for any $1\leqslant i\leqslant u$, 
$(o_{i0}'^2+o_{i2}'^2)\eta _i\in\tilde\eta O_0+2\tilde\eta O$ and 
$(o_{i1}'^2+o_{i3}'^2)\eta _i\in 2\tilde\eta O$. This proves 
the proposition. 
\end{proof}
\medskip

\subsection{Introducing $\C M=\C M(G)\in\MF ^e_S$} 
\label{S3.10} 

Choose a rings identification  
$\kappa _{S'O'}:S'\operatorname{mod}t^e\To O'\operatorname{mod}2$ 
such that 
$\kappa _{S'O'}|_{S\operatorname{mod}t^e}=\kappa _{SO}$. Consider 
$l_0=\sum _ia_iX_i\in I_{B'}$ from Subsection \ref{S3.9} and set 
$$l_1=\sum _ib_iX_i\equiv 
(-l_0-\tilde\eta 'l_0^2-\dots -\tilde\eta '^{2^n-1}l_0^{2^n}- 
\dots )\operatorname{mod}2(\tilde\eta 'I_{B'})^{loc}$$
Let $\tilde s',\alpha _1,\dots ,\alpha _u\in S'$ be such that 
$\kappa _{S'O'}(\tilde s'\operatorname{mod}t^e)=\tilde\eta '\operatorname{mod}2$ and 
$\kappa _{S'O'}(\alpha _i\operatorname{mod}t^e)=b_i\operatorname{mod}2$ 
for $1\leqslant i\leqslant u$. 

Introduce $\C M'=\C M'(G)=(M'^0, M'^1,\varphi _1')\supset\C N':=\C N\otimes _SS'$ 
such that $M'^0=m'S'\oplus (N^0\otimes _SS')$, 
$M'^1=m'^1S'+(N^1\otimes _SS')\subset M'^0$ with 
$m'^1=\tilde s'm'+\sum _i\alpha _in_i$, and $\varphi _1'(m'^1)=m'$. 

\begin{Prop} 
\label{P3.13} 
 {\rm a)}\ $\C M'\in\MF ^{2e}_{S'}$;
\medskip 

{\rm b)}\ there is $\C M=\C M(G)\in\MF ^e_S$ such that $\C M'=\C M\otimes _SS'$.
\end{Prop}

\begin{proof} a) $\C M'\in\MF ^{2e}_{S'}$ means that $t^em'\in M'^1$, i.e. 
for $1\leqslant i\leqslant u$, one has  
 $t^e\tilde s^{-1}\alpha _i\equiv 0\operatorname{mod}
\tilde s_i$ or, equivalently, 
$$(t^e\tilde s_i^{-1})\alpha _i\equiv 0\operatorname{mod}\tilde s.$$
 Using (\ref{AAE3}) and the identification $\kappa _{S'O'}$ 
we can rewrite these congruences in the form 
$\left (\sum _lo'_{il}\right )\eta _i'\equiv 0\operatorname{mod}\tilde\eta $ 
or equivalently, $\left (\sum _lo^{\prime 2}_{il}\right )
\eta _i\equiv 0\operatorname{mod}\tilde\eta ^2$. Clearly, 
these conditions follow from Proposition \ref{P3.12}. 

b) The criterion of the existence of a descent $\C M$ of $\C M'$ to $S$ 
from Proposition \ref{P1.3} can be specified in our case in the following form:
for all $1\leqslant i\leqslant u$, 
$\alpha _i\in S\operatorname{mod}\tilde s_iS'$.

Using the identification $\kappa _{S'O'}$ we rewrite these conditions in the form 
$\tilde b_i\operatorname{mod}\tilde\eta _i'\in O\operatorname{mod}\tilde\eta '$.  
By (\ref{AAE3}) we can replace them by 
$\sum _lo'_{il}\in O\operatorname{mod}\tilde\eta _i'\tilde\eta '$, then by 
$\sum _lo_{il}'^2\in O_0\operatorname{mod}\tilde\eta _i\tilde\eta $ and finally by 
$(\sum _lo'^2_{il})\eta _i\in O_0\operatorname{mod}2\tilde\eta O$. 
But this is given again by Proposition \ref{P3.12}. 
\end{proof}

\subsection{Construction of isomorphism $\C G_O(\C M)\simeq G$} 
\label{S3.11} 

We know that  $G=\Spec A$, where $A=B[X]$, 
$[2](\tilde\eta X)=[2](f)\in (\tilde\eta ^2I_{B_0}^{loc}\subset 
(\tilde\eta ^2I_B)^{loc}$ 
and $\delta _{LT}(\tilde\eta X)=\delta _{LT}(f)\in (\tilde\eta I_{B_0\otimes B_0})^{loc}
\subset (\tilde\eta I_{B\otimes B})^{loc}$.

Let $A'=A\otimes _OO'$ and let  
$Z\in I_{A'}$ be such that $\tilde\eta 'Z=
[\tilde\eta X]-\sum _{i,l}[o'_{il}X_i]$. Note that 
\begin{equation} \label{E3.7} 
 Z+\tilde\eta '^{-1}\sum _{i,l}o'_{il}X_i\equiv 
Z-l_0\equiv Z+l_1\equiv 0\operatorname{mod}\tilde\eta 'I_{A'}
\end{equation}

\begin{Prop} 
\label{P3.14} 
{\rm a)} $\tilde\eta 'Z\in (\tilde\eta 'I_{A'}(2))^{loc}$;
\newline 
{\rm b)} $\delta ^+(Z-l_0)\equiv \tilde\eta ' I_{A'\otimes A'}(4)$;
\newline 
{\rm c)}$\tilde\eta ' (Z^2/2)\equiv Z+l_1\operatorname{mod} 
(\tilde\eta 'I_{A'}(4))^{loc}$.
\end{Prop}

\begin{proof} a) By Proposition \ref{P3.11}a), 
$[2](\tilde\eta 'Z)=[2](h)\equiv 2\tilde\eta 'l_0\operatorname{mod}
2(\tilde\eta I_{B'})^{loc}$. 
This implies $(\tilde\eta 'Z)^2\in 2(\tilde\eta 'I_{A'})^{loc}$, i.e. 
$\tilde\eta ^{\prime 1/2}Z\in I_{A'}(2)^{loc}$. If 
$\tilde\eta\in O_0^*$ then a) is proved. Continue with the following 
congruence 
$$l_{LT}(\tilde\eta 'Z)=l_{LT}(h)\equiv 
\tilde\eta 'l_0\operatorname{mod}(\tilde\eta I_{B'}(4))^{loc}$$
(cf. Proposition \ref{P3.10}). It implies that 
\begin{equation}\label{E3.8} 
Z-l_0+\tilde\eta '(Z^2/2)(1+\tilde\eta Z^2/2)\in (\tilde\eta 'I_{A'}(4))^{loc}.
\end{equation} 
By \eqref{E3.7} $(Z-l_0)/\tilde\eta '\in I_{A'}$ and 
$1+\tilde\eta (Z^2/2)\equiv 1\operatorname{mod}\tilde\eta 'I_{A'}$ 
is invertible 
(use that $\tilde\eta \notin O_0^*$). Therefore, 
$Z^2/2\in I_{A'}$ and a) is completely proved.

b) By proposition \ref{P3.11}b),  
$\delta _{LT}(\tilde\eta 'Z)\equiv \tilde\eta '\delta ^+l_0\operatorname{mod}
\tilde\eta I_{B'\otimes B'}(4)$. It remains to note that $Z\in I_{A'}(2)$ implies that 
$$\delta _{LT}(\tilde\eta 'Z)\equiv\delta ^+(\tilde\eta 'Z)
\operatorname{mod}\tilde\eta I_{A'\otimes A'}(4).$$

c)\ Iterating (\ref{E3.8}) we obtain 
$$\tilde\eta '(Z^2/2)\equiv Z-l_0-\tilde\eta '(\tilde\eta 'Z^2/2)^2\equiv 
Z-l_0-\tilde\eta 'l_0^2-\tilde\eta '^3(\tilde\eta 'Z^2/2)^4$$
$$\equiv Z-l_0-\tilde\eta 'l_0^2-\dots -\tilde\eta '^{2^n-1}l_0^{2^n}-\dots 
\equiv Z+l_1\operatorname{mod}\,(\tilde\eta 'I_{A'}(4))^{loc}.$$
The proposition is proved. 
\end{proof} 

Consider the ideals $J_{A'}$ and $J_{A'\otimes A'}$,  
cf. Subsection \ref{S1.3}. 

\begin{Cor}\label{C3.15} 
There is $\wt{Z}\in I_{A'}$ such that 

{\rm a)} $\wt{Z}\equiv (\tilde\eta 'I_{A'}(4))^{loc}$;

{\rm b)} $\tilde\eta '(\wt Z^2/2)\equiv\wt{Z}+l_1
\operatorname{mod}\tilde\eta 'J_{A'}$; 

{\rm c)} 
 $\delta ^+(\wt{Z}), \delta ^+(\wt{Z}^2/2)\in J_{A'\otimes A'}$.
\end{Cor}

\begin{proof} 
 Let $\wt{Z}=\tilde\eta '(Z^2/2)-l_1$.

Part a) follows directly from Proposition \ref{P3.14}. 

Because $(\tilde\eta 'I_{A'}(4))^{loc}\subset I_{A'}(2)$ this implies that 
\begin{equation}\label{Ee}
 \wt{Z}^2/2\equiv Z^2/2\operatorname{mod}J_{A'}
\end{equation}
and we obtained part b). 

By Proposition \ref{P3.14}b), $\delta ^+Z\in I_{A'\otimes A'}(4)$. This implies that 
$\delta ^+(Z^2/2)\in J_{A'\otimes A'}$, use Proposition \ref{P3.14}a). 
So, $\delta ^+Z\in J_{A'\otimes A'}$ and by \eqref{Ee} we obtain 
$\delta ^+(Z^2/2)\in J_{A'\otimes A'}$.  
\end{proof}

By above Corollary the correspondences 
$m'\mapsto {\wt Z}^2/2\operatorname{mod}J_{A'}$ and 
$m'^1\mapsto \wt{Z}\operatorname{mod}J_{A'}$ 
give a map of filtered modules $\c M'\To \iota (A')$. 
This gives a morphism of $O'$-algebras  
$\Pi ':A(\c G_{O'}(\c M'))\To A'$. By Proposition \ref{P1.11} 
we can assume that $\Pi '$ is also a morphism of coalgebras. 
Both these coalgebras contain $B'=B\otimes _OO'$ 
and $\Pi '|_{B'}$ is isomorphism. Similarly, both the coalgebras 
have  as their  quotient the coalgebra $A(\mu _{\eta })\otimes O'$ 
and $\Pi '$ induces 
on it a coalgebra isomorphism as well. This implies that 
$\Pi '$ is isomorphism of coalgebras and 
Theorem \ref{T3.1} is completely proved.

\end{document}